\documentclass[12pt]{article}
\usepackage{amsmath,amsfonts,amstext,amssymb,amsbsy,amsopn,amsthm,eucal}
\usepackage{dsfont}
\usepackage{graphicx}   

\begin{document}

\newtheorem{theorem}{Theorem}[section]
\newtheorem*{mtheorem}{Theorem}
\newtheorem*{mthm_orb}{Theorem}
\newtheorem*{mthm_gquot}{Theorem}
\newtheorem*{thm1}{Theorem}
\newtheorem*{thm2}{Theorem}
\newtheorem*{thm3}{Theorem}

\newtheorem{prop}{Proposition}[section]
\newtheorem*{proposition}{Proposition}
\newtheorem*{prop1}{Proposition 1}
\newtheorem*{prop2}{Proposition 2}
\newtheorem*{prop3}{Proposition 3}

\newtheorem{lemma}{Lemma}[section]
\newtheorem*{lemma2}{Lemma 2}
\newtheorem*{lemma3}{Lemma 3}

\newtheorem{corollary}{Corollary}[section]
\newtheorem*{cor2}{Corollary 2}
\newtheorem*{cor3}{Corollary 3}
\newtheorem*{cor4}{Corollary 4}
\newtheorem*{cor5}{Corollary 5}

\theoremstyle{definition}
\newtheorem{definition}{Definition}[section]

\theoremstyle{remark}
\newtheorem{remark}{Remark}[section]

\theoremstyle{remark}
\newtheorem{example}{Example}[section]

\title{Geometric Structures of Collapsing Riemannian Manifolds I}

\author{Aaron Naber\thanks{Department of Mathematics, Princeton University,
        Princeton, NJ 08544 ({\tt anaber@math.princeton.edu}).} and Gang Tian\thanks{Department of Mathematics, Princeton University, Princeton, NJ 08544 ({\tt tian@math.princeton.edu}).}}

\date{\today}
\maketitle
\begin{abstract}
Let $(M^{n}_{i},g_{i},p_{i})$ be a sequence of smooth pointed complete n-dimensional Riemannian Manifolds with uniform bounds on the sectional curvatures and let $(X,d,p)$ be a metric space such that $(M^{n}_{i},g_{i},p_{i})\rightarrow (X,d,p)$ in the Gromov-Hausdorff sense. Let $\mathcal{O}\subseteq X$ be the set of points $x\in X$ such that there exists a neighborhood of $x$ which is isometric to an open set in a Riemannian orbifold and let $B=\mathcal{O}^{c}$ be the complement set.  Then we have the sharp estimates $dim_{Haus}(B)\leq min\{n-5,dim_{Haus}X -3\}$, and further for arbitrary $x\in X$  we have that $x\in \mathcal{O}$ iff a neighborhood of $x$ has bounded stratified curvature.  In particular, if $n\leq 4$ then $B=\emptyset$ and $(X,d)$ is a Riemannian orbifold.  Our main application is to prove that a collapsed limit of Einstein four manifolds has a smooth Riemannian orbifold structure away from a finite number of points, and that near these points the curvature has a $-dist^{-2}$ lower bound.
\end{abstract}

\section{Introduction}

This is the first in a series of papers where the authors seek to study and build structure for a collapsing sequence of Riemannian manifolds $(M^{n}_{i},g_{i})\stackrel{GH}{\rightarrow} (X,d)$ under a bounded curvature assumption and using this structure to study, among other things, the compactification of Einstein moduli spaces.  The study of such collapse has a rich history, see among others the foundational works of \cite{CG1} \cite{F} \cite{F1} \cite{CFG}, and in this first paper our primary focus is on metric orbifold structure of the limit space $X$.  Namely we ask the question: when and how often do points $p\in X$ have neighborhoods isometric to a Riemannian orbifold?  Work of this sort goes back to \cite{F1} where it was shown every point in $X$ has at least quotient singularity structure, and in \cite{F3} where it was further shown that if the $\{M_{i}\}$ collapse only one dimension then $X$ is in fact a Riemannian orbifold.  The basis of this theorem is that when the $\{M_{i}\}$ drop only one dimension it can be shown that the natural quotient singularity structure of $X$ must have finite isotropy at each point (see Theorem \ref{thm_main_collapse} for a generalization and a different proof of this).

In the general case, when the $\{M_{i}\}$ drop two or more dimensions, it need not be the case that the isotropy of the natural quotient action be finite.  However, this does not necessarily stop the quotient geometry on $X$ from having a further reduction to a Riemannian orbifold structure.  When a group $G$ acts on a manifold there are often local reductions of this group action, see Section \ref{sec_gtq} for a precise definition, which can be used to simplify the quotient structure.  By studying the type of group actions that can arise in the context of collapse the existence of these local reductions can be essentially classified, and so in turn can be used to understand when neighborhoods on $X$ have Riemannian orbifold structures.  It will turn out that orbifold singularities are more common than generic conic singularities.  It is worth pointing out that by applying the metric approximation results of \cite{A},\cite{AC},\cite{PWY} that many of the results of this paper also apply when $(M^{n},g_{i})$ have either uniform Ricci curvature bounds and lower bounds on the conjugacy radius or even only uniform lower Ricci bounds and a lower bound on the conjugacy radius (though in the latter case a derivative of regularity is lost on the Riemannian orbifold limit).

As an application we will use these results to study collapsing Einstein four manifolds.  Namely, using the $\epsilon$-regularity estimates of \cite{CT} we will show that the collapsed limit of a sequence of Einstein four manifolds is smooth Riemannian orbifolds away from a finite number of points with a distinct lower curvature bound.  In the sequel to this paper we will build additional structure on these limit spaces which will be useful in further understanding the topology and metric structure of the neighborhoods of these singularities.

Similar to \cite{CFG} we make the following definition for convenience.

\begin{definition}
Let $(M^{n},g,p)$ be a pointed $n$-dimensional Riemannian manifold and $\{A\}^{k}_{0}$ be a sequence of $k+1$ positive real numbers, for $0\leq k \leq \infty$.  We say $(M,g,p)$ is $\{A\}^{k}_{0}$-regular at $p$ if $\forall$ $r <\pi A_{0}^{-1/2}$ $B_{r}(p)$ has compact closure in $M$ with $|sec_{g}|\leq A^{0}$ and $|\nabla^{(j)} Rm_{g}|\leq A^{j}$ $\forall$ $1\leq j\leq k$ in $B_{r}(p)$ .  We say $(M,g)$ is $\{A\}^{k}_{0}$-regular if it is $\{A\}^{k}_{0}$-regular at every point.  If $(M,g)$ is a Riemannian manifold with boundary then we insist that $B_{r}(p)$ have compact closure in the interior of $M$.
\end{definition}

If a metric space $(X,d)$ is the Gromov-Hausdorff limit of a sequence of Riemannian manifolds which are $\{A\}^{k}_{0}$-regular we will often refer to it informally as a collapsed space.  Since we are interested in understanding the existence of points in collapsed limits with metric orbifold structures we introduce the following:

\begin{definition}\label{def_met_orb}
Let $(X,d)$ be a metric space.  For $0\leq k \leq \infty$ and $0 \leq \alpha < 1$ we call a point $p\in X$ a $C^{k,\alpha}$-orbifold point if there exists an $\epsilon>0$ such that $(B_{\epsilon}(p),d)$ is isometric to a Riemannian orbifold with a $C^{k,\alpha}$ metric; otherwise call it a non-$C^{k,\alpha}$ orbifold point.
\end{definition}
\begin{remark}
Notice our definition of orbifold and nonorbifold includes the existence of a metric on the finite cover and is more than just a topological condition.  For most of the theorems it is the existence of this metric that is the difficult part.  Also when the context is clear we will refer to a $C^{k,\alpha}$-orbifold point as simply a metric orbifold point and a non-$C^{k,\alpha}$ orbifold point as simply a nonorbifold point.
\end{remark}

Note that the set of metric orbifold points is an open set in $X$, hence the set of nonorbifold points is closed.  We also remark that if $X$ is a collapsed space then it follows from \cite{F1} that $X$ has a well defined Hausdorff dimension.

Our first main theorem is

\begin{theorem}\label{thm_main}
Let $(M^{n}_{i},g_{i},p_{i})$ be a sequence of pointed n-dimensional Riemannian Manifolds that are $\{A\}^{k}_{0}$-regular and assume $(X,d)$ is a metric space such that $(M^{n}_{i},g_{i},p_{i}) \stackrel{GH}{\rightarrow} (X,d,p_{\infty})$.  Let $\mathcal{O}\subseteq X$ be the set of $C^{k+1,\alpha}$ orbifold points and  $B=\mathcal{O}^{c}\subseteq X$ be the compliment set of nonorbifold points in $X$.  Then $dim_{Haus}(B)\leq min\{n-5, dimX-3\}$.  In particular, if $n\leq 4$ then $B=\emptyset$ and $(X,d)$ is a $C^{k+1,\alpha}$ Riemannian orbifold.
\end{theorem}
\begin{remark}\label{rem_thm_main}
The result is purely local in that we will actually prove that if the sequence is only assumed $\{A\}^{k}_{0}$-regular at $p_{i}$, then in a definate neighborhood of $p_{\infty}$ the result still holds.
\end{remark}

We do in fact have something a bit stronger, that in dimension $5$ the nonorbifold singular set $B$ will be a collection of isolated points, and in dimension $n\geq 6$ that $B$ will have a stratified structure of at most dimension $n-5$ ($dimX-3$).  These dimensional bounds are sharp as can be seen from Example \ref{exam_sharp_est}, and it is not enough to assume only a bound on the Ricci tensor as can be seen from Example \ref{exam_K3_Collapse} (though as remarked previously if you make the further assumption of a lower bound on the conjugacy radius then this is enough).

To characterize those points in a collapsed space which are metric orbifold we introduce the following notation:

\begin{definition}
Let $(X,g)$ be a Riemannian stratified space.  If $x\in X$ we say the curvature at $x$ is stratified bounded if there exists a $C>0$ and a neighborhood $U$ of $x$ such that for any strata $S$ restricted to $U$, the curvature of $(S,g|_{S})$ is bounded by $C$.
\end{definition}

With this in hand we get:

\begin{theorem}\label{thm_main2}
Let $(M^{n}_{i},g_{i},p_{i})$ be a sequence of pointed n-dimensional Riemannian Manifolds that are $\{A\}^{k}_{0}$-regular, $k\geq 1$, and assume $(X,d)$ is a metric space such that $(M^{n}_{i},g_{i},p_{i}) \stackrel{GH}{\rightarrow} (X,d,p_{\infty})$.   If we let $p\in X$ be an arbitrary point then $p$ is a $C^{k+1,\alpha}$- orbifold point iff the curvature at $p$ is stratified bounded.
\end{theorem}
\begin{remark}
In fact it is enough in the above to assume only that the curvature on the open dense manifold part of a neighborhood of $p$ is uniformly bounded.  Also when $k=0$ something may be said, but one must consider the alexandroff curvature of each strata in a neighborhood of $p$.
\end{remark}

The above tells us in particular that in the category of collapsed metric spaces that if $(X,d)$ has stratified bounded curvature then $(X,d)$ is a Riemannian orbifold.  As an application we obtain:

\begin{theorem}\label{thm_main3}
Let $(M_{i},g_{i},p_{i})$ be a sequence of compact Einstein Four Manifolds with Einstein constants $|\lambda_{i}|\leq 3$ and Euler numbers $\chi(M_{i})\leq D$.  Then after passing to a subsequence there exists a metric space $(X,d,p)$ and points $\{q_{j}\}_{1}^{N}\in X$, with $N\leq N(D)$, such that $(M_{i},g_{i},p_{i})\stackrel{GH}{\rightarrow}(X,d,p)$ and if $x\in X-\bigcup q_{j}$ then a neighborhood of $x$ is isometric to a smooth Riemannian Orbifold.  Further there exists a universal constant $C$ such that for each $x\in X$ we have the curvature estimate $sec(x) \geq min\{-1,-C d(x,\{q_{j}\})^{-2}\}$
\end{theorem}

It is seen from the work of \cite{GW}, see example \ref{exam_K3_Collapse}, that unlike the noncollapsing case there can be points in an Einstein limit $X$ which are nonorbifold and have curvature blow up.

A brief outline of the paper is as follows.  In the Section \ref{sec_col} we do some analysis, based on the work of \cite{CFG}, to study the local quotient singularity structure of $X$.  Namely we wish to see that not only can $X$ be locally realized as a quotient geometry, but that the isotropy of this action has a dimensional bound based on the amount $X$ collapses.  With $X$ locally written as a quotient geometry Sections \ref{sec_gtq} and \ref{sec_gtq2} seek to classify when the geometric quotient structure of Section \ref{sec_col} can be reduced to a finite geometric quotient.  In particular in Section \ref{sec_gtq} it is shown that away from a subset of codimension $3$ in $X$ every point has such a reduction, and in Section \ref{sec_gtq2} it is shown that a point in $X$ having such a reduction is equivalent to a bound on the curvature of $X$ (in the stratified sense).  Most of these two sections are centered on analyzing group quotients and many of the results also work for arbitrary such quotients.  Section \ref{sec_main1} uses these results to prove the first two main theorems.  The proof of Theorem \ref{thm_main3}, which is a straight forward application of Theorem \ref{thm_main} and the results of \cite{CT}, is done in Section \ref{sec_main2}.  In Appendix \ref{sec_gq} some basic information about quotient geometries which is used throughout the paper is presented.  An important tool used in the paper is that distance preserving maps between Riemannian orbifolds are smooth in the orbifold sense, and this is proved in Appendix \ref{sec_orb}.

We begin with some examples:

\begin{example}
Consider $T^{3}=S^{1}\times S^{1}\times S^{1}$ with a metric $g_{\epsilon} = d\theta^{2}_{1} + d\theta^{2}_{2} + \epsilon d\theta^{2}_{3}$.  Let $\mathds{Z}_{2}$ act on $T^{3}$ by $(e^{i\theta_{1}}, e^{i\theta_{2}}, e^{i\theta_{3}})\mapsto (e^{i\theta_{1}}, e^{-i\theta_{2}}, -e^{i\theta_{3}})$.  Then $\mathds{Z}_{2}$ acts freely and isometrically and by taking a quotient we get the Riemannian Manifolds $(T^{3}/\mathds{Z}_{2},g_{\epsilon})$.  If we let $\epsilon \rightarrow 0$ then $(T^{3}/\mathds{Z}_{2},g_{\epsilon})$ Gromov-Hausdorff converges to a cylinder with boundary.
\end{example}

\begin{example}
Consider $\mathds{R}^{3}\times S^{1}$ and the $T^{2}$ action by $(e^{i\phi_{1}},e^{i\phi_{2}}) \cdot (x,y,z,e^{i \theta}) \mapsto (x,cos\phi_{1} y+sin\phi_{1} z, -sin\phi_{1} y +cos\phi_{1} z, e^{i(\phi_{2}+\theta)})$.  Then this is a faithful $T^{2}$ action and as in \cite{CG1} we can then pick a sequence of metrics $g_{\epsilon}$ on $\mathds{R}^{3}\times S^{1}$ which will collapse with bounded curvature such that $(\mathds{R}^{3}\times S^{1}, g_{\epsilon})\rightarrow \mathds{R}^{3}/S^{1}$, where $S^{1}$ acts on $\mathds{R}^{3}$ by $e^{i\phi}\cdot(x,y,z) = (x,cos\phi y+sin\phi z, -sin\phi y +cos\phi z)$ and hence is isometric to the half plane $\mathds{R}^{2}/\mathds{Z}_{2}$.  The important point of this simple example is that although the isotropy of the $S^{1}$ action on $\mathds{R}^{3}$ jumps in dimension on the $x$-axis, the quotient still has a smooth Riemannian orbifold structure.

We can construct a similar sequence in the compact category by considering $S^{3}\times S^{1}$.  Then if $p_{n}$, $p_{s}$ are the north and south poles of $S^{3}$, then we can write $S^{3}-\{p_{n},p_{s}\} \approx \mathds{R}\times S^{2}$.  Then we can construct a smooth $S^{1}$ action on $S^{3}$ by rotating the $S^{2}$ factor.  Again this gives a faithful $T^{2}$ action on $S^{3}\times S^{1}$ and we can collapse with bounded curvature to $S^{3}/S^{1}$, which is isometric to the half sphere $S^{2}/\mathds{Z}_{2}$.
\end{example}

\begin{example}\label{exam_sharp_est}
Consider $\mathds{R}^{4}\times S^{1}$.  Writing $\mathds{R}^{4}-\{0\} \approx \mathds{R}\times S^{3}$ and $S^{3} = \{(z_{1},z_{2})\in \mathds{C}^{2}: |z_{1}|^{2}+|z_{2}|^{2}=1\}$ we can let $S^{1}$ act on $S^{3}$ by $\lambda\cdot (z_{1},z_{2}) = (\lambda z_{1}, \lambda^{2} z_{2})$.  This gives a smooth action on $\mathds{R}^{4}-\{0\}$ which extends to a smooth action on $\mathds{R}^{4}$.  Thus we have a faithful $T^{2}$ action on $\mathds{R}^{4}\times S^{1}$ by letting the second $S^{1}$ factor act trivially on itself.  As before we can collapse along the orbits with bounded curvature to get $(\mathds{R}^{4}\times S^{1}, g_{\epsilon})\rightarrow (X,d)$, where now $X$ is a cone over the teardrop orbifold.  Away from the origin $X$ is itself a smooth Riemannian Orbifold, but the curvature blows up as one approaches the origin and thus $X$ has a true conic metric singularity at an isolated point.  This shows us that Theorem \ref{thm_main} is sharp.  Again we can do a similar construction in the compact category by considering $S^{4}\times S^{1}$.
\end{example}

\begin{example}\label{exam_K3_Collapse}
From \cite{GW} one can construct a sequence of $K3$ surfaces $(M^{4},g_{i})$ where $g_{i}$ are Ricci-Flat and $(M^{4},g_{i})\stackrel{GH}{\rightarrow} (S^{2},d)$, where $S^{2}$ is the topological sphere and $d$ is a distance function induced by a metric $g_{\infty}$ which is smooth away from precisely $24$ points $\{q_{j}\}^{24}_{1}\in S^{2}$.  Near these points the curvatures of the smooth metric tend to $\infty$ at a rate bounded by  $o(d^{-2})$.
\end{example}

It is also interesting to point out that if we consider the closed unit disk $\bar D$ in $\mathds{R}^{2}$ that Theorem \ref{thm_main2} tells us that $\bar D$ can not be approximated arbitrarily closely in the Gromov Hausdorff sense by Riemannian manifolds with any uniform bounds on the curvature and dimension.  This follows because although $\bar D$ is certainly a topological orbifold which is flat as an Alexandroff space (and in the stratified sense), it is easy to check that it is not a Riemannian orbifold.

\section{Structure of Collapsed Spaces}\label{sec_col}

The goal of this section is to prove some slight refinements of theorems from \cite{F1}, \cite{F3}, \cite{CFG}.  The type of structure introduced in the next theorem was first observed in \cite{F1} and later refined in \cite{CFG} where it was used to study collapsed regions of individual Riemannian manifolds.  The proof of the first part of the following is as in \cite{CFG}, however we will go through it both for convenience and because we wish to understand the structure of the isotropy group limit.

\begin{theorem}\label{thm_main_collapse}
Let $(M^{n}_{i},g_{i},p_{i})$ be smooth $n$-dimensional Riemannian Manifolds that are $\{A\}^{k}_{0}$-regular at $p_{i}$.  Assume $(M_{i},g_{i},p_{i})\stackrel{GH}{\rightarrow} (X,d,p_{\infty})$, a metric space.  Then there exists $\epsilon>0$ such that for all $i$ sufficiently large there are neighborhoods $B_{2\epsilon}(p_{i})\subseteq U_{i} \subseteq M_{i}$ with the properties

\begin{enumerate}
\item If $\tilde{U}_{i}$ is the universal cover of $U_{i}$ and $\tilde{p}$ is any lift of a point $p\in B_{\epsilon}(p_{i})$, then $inj_{\tilde U_{i}}(\tilde p) > \epsilon$.

\item If $\Lambda_{i}$ are the fundamental groups of $U_{i}$ and $\tilde p_{i}\in\tilde U_{i}$ is a lift of $p_{i}$ then, after passing to a subsequence, $(\tilde U_{i},\tilde g_{i}, \Lambda_{i}, \tilde p_{i})\stackrel{eGH - C^{k+1,\alpha}}{\rightarrow}(\tilde U_{\infty}, \tilde g_{\infty}, N, \tilde p_{\infty})$ where $N$, a closed subgroup of the isometry group which acts properly on $(\tilde U_{\infty},\tilde g_{\infty})$, is a finite extension of a connected nilpotent Lie Group.

\end{enumerate}

Further, if $inj_{M_{i}}(p_{i})$ is not uniformly bounded from below and $I_{\tilde p_{\infty}}$ is the isotropy subgroup of $N$ at $\tilde p_{\infty}$ then $dim I_{\tilde p_{\infty}} < dim N$.
\end{theorem}
\begin{remark}\label{rem_col}
The neighborhoods $U_{i}$ are related to the $(\rho,k)$-\textit{round} neighborhoods of \cite{CFG}, though while the $(\rho,k)$-\textit{round} neighborhoods were meant to capture all sufficiently collapsed directions of $M_{i}$ the neighborhoods $U_{i}$ are only meant to capture those directions which \textit{continue} to collapse in the sequence.  This is why $N$ is only a finite extension of its identity component and the reason that the $\epsilon$ variable in the theorem depends on the collapsed limit $X$ itself and not just on the regularity properties of the $(M_{i},g_{i})$.  Additionally one can take $U_{\infty}/N$ to be contractible in the above.  If one is willing to let $N$ be a not necessarily finite extension of a connected nilpotent, or for $U_{\infty}/N$ to not be contractible, then $\epsilon$ can be taken to depend only on $A^{0}$ and $n$.  The $C^{k+1,\alpha}$ convergence is in the Cheeger-Gromov sense (see \cite{P} for instance) and for a rigorous definition of equivariant Gromov Hausdorff convergence (abbreviated eGH) see \cite{F3} or \cite{CFG}.
\end{remark}

The dimension estimate on the isotropy $I$ can be viewed as a generalization of a theorem in \cite{F3} where this is proved under the assumption that $M_{i}$ collapse only one dimension.  In particular a corollary of the above is the following:

\begin{corollary}[Fukaya \cite{F3}]
Let $(M^{n}_{i},g_{i},p_{i})$ be complete $n$-dimensional Riemannian Manifolds that are $\{A\}^{k}_{0}$-regular at $p_{i}$.  Assume $(M_{i},g_{i},p_{i})\stackrel{GH}{\rightarrow} (X,d,p_{\infty})$, a metric space with $dim_{Haus} X = n-1$.  Then $X$ is a Riemannian orbifold.
\end{corollary}
\begin{proof}
By Theorem \ref{thm_main_collapse} $X$ is locally an isometric quotient of a Riemannian manifold by a proper and faithful Lie group action $N$.  Since $dim X=n-1$ we see that $dim N=1$, and since the dimension of the isotropy is strictly less than that of $N$ we see that $dim I_{p} = 0$ for every $p\in X$.  Hence by Corollary \ref{cor_geom_slice} we have the result.
\end{proof}

We begin with a review of some of the basic constructions in \cite{F1}.

\begin{definition}\label{def_frame_metr}
Let $(M,g)$ be a Riemannian Manifold, possibly with boundary.  We define the associated Riemannian Manifold $(FM,h)$ where $FM$ is the $O(n)$-frame bundle of $M$ and $h$ is the canonical metric defined by using the Levi-Civita connection to define a horizontal distribution, letting each fiber be isometric to the standard bi-invariant metric on $O(n)$, and assuming the projection map is a Riemannian submersion.
\end{definition}

The following are easy to check

\begin{lemma}\label{lem_frame_reg}
Let $(M,g)$ be $\{A\}^{k}_{0}$-regular, then $(FM,h)$ is $\{B\}^{k-1}_{0}=\{B(A,n)\}^{k-1}_{0}$-regular.
\end{lemma}

\begin{lemma}
If $(M_{i},g_{i})$ is a precompact sequence with respect to the Gromov Hausdorff distance then $(FM_{i},h_{i},O(n))$ is precompact with respect to the equivariant Gromov Hausdorff topology.
\end{lemma}

The first lemma is a direct computation while the second follows from an easy $\epsilon$-density argument (see \cite{BBI} for standard precompactness results).  It follows that if $(M_{i},g_{i})\stackrel{GH}{\rightarrow}(X,d_{X})$ then, after possibly passing to a subsequence, there exists $(Y,d_{Y},O(n))$ such that $(FM_{i},h_{i},O(n))\stackrel{eGH}{\rightarrow}(Y,d_{Y},O(n))$ and hence $(Y,d_{Y})/O(n) \approx (X,d_{X})$.  In fact we have the following

\begin{lemma}[\cite{F1}]\label{lem_frame_lim}
Let $(M^{n}_{i},g_{i},p_{i})$ be smooth complete $n$-dimensional Riemannian Manifolds that are $\{A\}^{\infty}_{0}$-regular at $p_{i}$ and let $r < \frac{\pi}{\sqrt{A^{0}}}$.  Assume $(B_{r}(p_{i}),g_{i})\stackrel{GH}{\rightarrow} (X,d)$, a metric space.  Then $(FB_{r}(p_{i}),h_{i},O(n))\stackrel{eGH}{\rightarrow}(Y,h_{\infty},O(n))$ where $(Y,h_{\infty})$ is a smooth Riemannian manifold which is $\{C\}^{\infty}_{0} = \{C(A,n)\}^{\infty}_{0}$-regular.
\end{lemma}

See \cite{F1} for details.

The following simple lemma will be useful for studying the injectivity radius of a Riemannian manifold.

\begin{lemma}\label{lem_inj}
Let $(M,g)$ be a Riemannian manifold with $|sec|\leq 1$, $p\in M$ and $r<\pi$ such that $B_{r}(p)$ has compact closure in $M$.  Then if it holds that for every closed curve $\gamma$ at $p$ with $|\gamma| < r$ we have a one parameter family of closed curves $\gamma_{t}$ with $|\gamma_{t}|<2r$, $\gamma_{0}(s)=\gamma(s)$ and $\gamma_{1}(s)=p$, then $inj_{M}(p)\geq r/2$.
\end{lemma}
\begin{proof}
The map $exp_{p}:B_{r}(0)\rightarrow B_{r}(p)$ is certainly a local diffeomorphism.  Assume $x,y\in B_{r/2}(0)$ is such that $exp_{p}(x)=exp_{p}(y)$, and consider the closed curve which is the adjoinment $\gamma = \gamma_{x}\sqcup\gamma_{y}$ where $\gamma_{x}(s)=exp_{p}(s\cdot x)$ and $\gamma_{y}(s)=exp_{p}(s\cdot y)$.  By assumption there exists $\gamma_{t}$ with $|\gamma_{t}|< 2r$ and $\gamma_{1}(s)=p=\gamma_{t}(1)=\gamma_{t}(0)$, and hence $Image(\gamma_{t}(s))\subseteq exp_{p}(B_{r}(0))$.  But by simply reparametrizing the interval $[0,1]\times[0,1]$ we may view this as a one parameter family of curves $\tilde\gamma_{t}(s)$ with $\tilde\gamma_{0}=\gamma_{x}$, $\tilde\gamma_{1}=\gamma_{y}$, $\tilde\gamma_{t}(0)=p$ and $\tilde\gamma_{t}(1)=\gamma_{x}(1)$.  Since we also have $Image(\tilde\gamma_{t}(s))\subseteq exp_{p}(B_{r}(0))$ we can lift uniquely with $\gamma_{t}(0)$ lifting to $0$.  However then $\gamma_{0}(s)$ must lift to $s\cdot x$, $\gamma_{1}(s)$ to $s\cdot y$, and $\gamma_{t}(1)\in exp_{p}^{-1}(p)$.  Since $exp$ is a local diffeomorphism this tells us $x=y$.  Hence the exponential map is one to one as claimed.
\end{proof}

The above is of course closely related to Gromov's notion of the pseudofundamental group, and our main application of it is to prove the following lemma:

\begin{lemma}\label{lem_inj_sub}
Let $(M,g)$ be a Riemannian manifold, $S\subseteq M$ a complete submanifold and $r<\pi$ such that $B_{r}(S)$ has compact closure in $M$.  Assume $|sec_{M}|, |T_{S}|\leq 1$, where $T_{S}$ is the second fundamental form of $S$.  Then there exists a universal $C>0$ such that if $p\in S$ with $inj_{S}(p) > r$ and the normal injectivity radius of $S$ in $M$ satisfies $inj_{M}(S)>r$, then $inj_{M}(p)>C r$.
\end{lemma}
\begin{proof}
By scaling it is enough to assume $r=1$.  So let $\gamma$ be a closed curve at $p$ in $M$ of length $|\gamma|\leq l$, $l$ to be chosen.  If $l<2$ then because of the normal injectivity radius bound on $S$ we can use the normal exponential map $exp_{S}$ to define $\gamma_{t}(s)=exp_{S}((1-2t)exp^{-1}_{S}(\gamma(s)))$.  Hence $\gamma_{0}=\gamma$, $\gamma_{1/2}\subseteq S$ and because the curvatures of $M$ and second fundamental form of $S$ are uniformly bounded we have that if $l$ is sufficiently small then $|\gamma_{t}|\leq \frac{4}{3}|\gamma| \leq \frac{4}{3}l$.  Now $\gamma_{1/2}$ is a closed curve in $S$ at $p$, and so since we have a lower injectivity radius bound in $S$ at $p$ we have that if $|\gamma_{1/2}| < 2$ then we can similarly use the exponential map in $S$ at $p$ to define $\gamma_{t}(s)=exp_{p}((2-2t)exp^{-1}_{p}(\gamma_{1/2}(s)))$.  The curvature of $S$ with the induced metric is uniformly bounded because the second fundamental form is uniformly bounded, so again for $l$ small we then have that for $t\in[\frac{1}{2},1]$ that $|\gamma_{t}|\leq \frac{4}{3}|\gamma_{1/2}|< 2|\gamma| \leq 2 l$.  Since $\gamma$ was arbitrary up to the bound on its length we have by the Lemma \ref{lem_inj} that there is a uniform lower bound on the injectivity radius of $M$ at $p$.
\end{proof}

Now we can prove the main result of this section:

\begin{proof}[Proof of Theorem \ref{thm_main_collapse}]
Let $(M_{i},g_{i},p_{i})\stackrel{GH}{\rightarrow} (X,d,p_{\infty})$ be as in the statement of the theorem.  First we prove Theorem \ref{thm_main_collapse} $(1)$ and find $\epsilon>0$ and $U_{i}\supseteq B_{2\epsilon}$ with the desired lower injectivity radius bounds on $\tilde U_{i}$.  So assume no uniform $\epsilon$ exists, then in particular after passing to a subsequence we can assume the statement fails for $\epsilon_{i}=\frac{1}{i}$.  For $i$ sufficiently large we will show a contradiction.

Now for $r=\frac{1}{\sqrt A^{0}}$ we have that $(B_{3r}(p_{i}),g_{i})\stackrel{GH}{\rightarrow} (B_{3r}(p_{\infty}),d)$.  If we fix some $\delta = \frac{r}{100}$ we can let $\tilde g_{i}$ be smooth metrics on $B_{2r}(p_{i})$ as in \cite{A} such that $e^{-\delta}g_{i}<\tilde g_{i}<e^{\delta}g_{i}$, $|\nabla^{g_{i}}-\nabla^{\tilde g_{i}}|_{C^{0}}< \delta$, $|Rm[\tilde g_{i}]|\leq 2A^{0}$, and $|\nabla^{j}Rm[\tilde g_{i}]| < C(n,j,\delta)$ for $\forall$ $j$.  Let $(FB_{2r}(p_{i}),\tilde h_{i})$ be the frame bundle above $B_{2r}(p_{i})$ with metric $\tilde h_{i}$ as in Definition \ref{def_frame_metr} with respect to the perturbed metric $\tilde g_{i}$ on $B_{2r}(p_{i})$.  By Lemmas \ref{lem_frame_reg} and  \ref{lem_frame_lim} we have that there exists $\{B\}_{0}^{\infty}=\{B(n,A^{0})\}_{0}^{\infty}$ such that $((FB_{2r}(p_{i}),\tilde h_{i})$ is $\{B\}_{0}^{\infty}$-regular and such that after possibly passing to a subsequence we have $(FB_{2r}(p_{i}),\tilde h_{i},O(n))\stackrel{eGH}{\rightarrow}(Y_{\infty},\tilde h_{\infty},O(n))$, where $(Y_{\infty},\tilde h_{\infty})$ is also a $\{B\}_{0}^{\infty}$-regular Riemannian manifold.

Now let $(Y_{i},\tilde h_{i}) = ((FB_{r}(p_{i}),\tilde h_{i})$ be the portion of the frame bundle above $B_{r}(p_{i})$ and let $\iota < min\{r,inj(Y_{\infty})\}$.  Then by mollifying equivariant Gromov-Hausdorff maps between $Y_{i}$ and $Y_{\infty}$ it follows from \cite{CFG} that for large $i$ there exist $O(n)$-equivariant maps $f_{i}:Y_{i}\rightarrow Y_{\infty}$ which are almost Riemannian submersions such that there exists $C=C(n,B,\iota)$ such that $|\nabla^{2}f_{i}|\leq C$.  In particular for all $i$ sufficiently large we may assume $\frac{1}{C}|X|\leq |df_{i}[X]| \leq C|X|$ for all $x\in Y_{i}$ and for all horizontal vectors $X$ (recall that a horizontal vector at $x\in Y_{i}$ is by definition one which is perpendicular to the fiber tangent of $f^{-1}_{i}(f_{i}(x))$).  We also have the property that the level sets of $f_{i}$ are connected with diameter tending to zero as $i\rightarrow\infty$.  Let $\tilde p_{i}\in Y_{i}$ be a lift of $p_{i}$ and $\tilde p_{\infty}=lim_{i\rightarrow\infty}f_{i}(\tilde p_{i})\in Y_{\infty}$ (after passing to a further subsequence).

Now the second fundamental form of $O(n)\cdot\tilde p_{\infty} \equiv \tilde S_{\infty}$ has some uniform bound $D>0$. Let us define $\tilde S_{i} \equiv f_{i}^{-1}(O(n)\cdot\tilde p_{\infty})$, which are also invariant under the $O(n)$ action on $Y_{i}$.  Because of our bounds on $f_{i}$ and because $\tilde S_{\infty}$ has uniform bounds on its second fundamental form it is a small calculation to check that, after possibly modifying $D$, the second fundamental form of $\tilde S_{i}$ is also uniformly bounded by $D$ with respect to the metric $\tilde h_{i}$.

Now let $\iota_{0}=min\{r,inj_{Y_{\infty}}(\tilde S_{\infty})\}$, where $inj_{Y_{\infty}}(\tilde S_{\infty})$ is the normal injectivity radius of the surface $\tilde S_{\infty}$.  The claim is that there exists $\epsilon_{0} = \epsilon_{0}(n,D,\iota_{0})$ such that the normal injectivity radius of $\tilde S_{i}$ satisfies $inj_{Y_{i}}(\tilde S_{i})>\epsilon_{0}$.  To see this notice that because the curvature of $Y_{i}$ is bounded by $B^{0}$ and the second fundamental form of $\tilde S_{i}$ is bounded by $D$ that if $inj(\tilde S_{i})$ is sufficiently small, depending on just $B^{0}$ and $D$, that there must be a geodesic $\tilde\gamma$ with length $|\tilde\gamma|=l$ small, such that $\tilde\gamma(0),\tilde\gamma(l)\in \tilde S_{i}$ and $\dot{\tilde\gamma}(0),\dot{\tilde\gamma}(l)$ are horizontal to $\tilde S_{i}$.  If $\gamma(t) \equiv f_{i}(\tilde\gamma(t))$ is the projected curve in $Y_{\infty}$ then because $|\dot{\tilde\gamma}| (t)=1$ we have $|\dot\gamma|(t)\leq C$, and since $\dot{\tilde\gamma} (0)$ is perpendicular to $\tilde S_{i}$ we have that $|\dot\gamma|(0)\geq 1/C$ and is nearly horizontal to $O_{f_{i}(\tilde p_{i})}$.  If we let $\phi(t)=d_{\tilde h_{\infty}}(\tilde S_{\infty},\gamma(t))$ be the distance function to $\tilde S_{\infty}$, then if $l<2\iota_{0}$ we have that $\phi(t)$ is smooth for $t\in (0,l)$.  Our conditions on $\gamma(0)$ and $\dot\gamma(0)$ guarantee that $\phi(0)=0$ and $|\dot\phi|(0)\geq \tilde C=\tilde C(C)$, and because $\tilde\gamma$ is a geodesic in $Y_{i}$ and $|\nabla^{2}f_{i}|\leq C$ we have that we can alter $\tilde C$ such that the geodesic curvature of $\gamma$ satisfies $|\nabla_{\dot\gamma}\dot\gamma|(t)\leq \tilde C^{-1}$.  Because the curvature of $Y_{\infty}$ is bounded by $B^{0}$ and the second fundamental form of $\tilde S_{\infty}$ is bounded by $D$, we have the distance function to $\tilde S_{\infty}$ has uniform lower hessian bounds at each interior point $\gamma(t)$, and combining this with our geodesic curvature bounds on $\gamma(t)$ we see that we can pick $\tilde C$ such that $\frac{d^{2}}{dt^{2}}\phi(t)\geq -\tilde C$ for all $t\in (0,l)$.  In particular, since $\phi(0)=\phi(l)=0$ and $|\dot\phi|(0)\geq \tilde C$, we have a lower bound on $l$ by some $\epsilon_{0}$.

So far our construction has been on $Y_{i}$, so to descend to $M_{i}$ we now notice that since our smooth submanifold $\tilde S_{i}$ is $O(n)$ invariant it defines, through the $O(n)$ quotient map, a smooth submanifold $S_{i}$ of $M_{i}$.  Since $S_{i} = O(n) \backslash (O(n)\cdot f_{i}^{-1}(\tilde p_{\infty}))$ and $diam  f_{i}^{-1}(\tilde p_{\infty})\rightarrow 0$ as $i\rightarrow\infty$, we have that the diameter of $S_{i}$ tends to zero as $i$ tends to infinity. Since the quotient map $Y_{i}\stackrel{O(n)}{\rightarrow}B_{r}(p_{i})$ is a Riemannian submersion the second fundamental form bound on $\tilde S_{i}$ tells us the second fundamental form of $S_{i}$ is also bounded by $D$, and further that since $inj_{Y_{i}}(\tilde S_{i})>\epsilon_{0}$ we have that the normal injectivity radius of $S_{i}$ satisfies $inj_{M_{i}}(S_{i})>\epsilon_{0}$.  Now these estimates hold with respect to the perturbed metric $\tilde g_{i}$, but because $g_{i}$ and $\tilde g_{i}$ are $C^{1}$ close, closeness depending only on $A^{0}$,  we can modify $D$ and $\epsilon_{0}$ such that they hold for $g_{i}$ as well.  Now we define $U_{i} \equiv B_{\epsilon_{0}}(S_{i})$, and so $U_{i}$ is in fact diffeomorphic to a vector bundle over $S_{i}$ (compare \cite{F3}).  The uniform curvature and second fundamental form bounds on $S_{i}$ tells us the curvature of the induced metric on $S_{i}$ also has uniform curvature bounds.  Hence by \cite{R} we have that for large $i$ that $S_{i}$ is an infranil manifold, and in particular for some simply connected nilpotent Lie group $N_{i}$ we have that the fundamental group of $S_{i}$ satisfies $\Lambda_{i}\leq N_{i}\rtimes Aut(N_{i})$ with $|\Lambda_{i}/(\Lambda_{i}\cap N_{i})|\leq w(n)$.  Now let $\bar S_{i}$ be the lift of $S_{i}$ in $\tilde U_{i}$, which is itself isometric to the universal cover of $S_{i}$, and the claim is that there is a lower injectivity radius bound of $(\tilde U_{i},g_{i})$ in $B_{\epsilon_{0}/2}(\bar S_{i}) \subseteq \tilde U_{i}$.  To see this let $q\in \bar S_{i}$ be arbitrary, then if we can show a lower bound of $inj_{\tilde U_{i}}(q)$ at any such $q$ we will have established the claim because the curvature of $\tilde U_{i}$ are uniformly bounded and $B_{\epsilon_{0}}(\bar S_{i}) \subseteq \tilde U_{i}$.  Now it follows from \cite{BK} that for all $i$ sufficiently large that $\bar S_{i}$ also has a uniform injectivity radius bound by some $\epsilon_{0}$ (modify $\epsilon_{0}$).  But now we can just apply lemma \ref{lem_inj_sub} to see that there is a lower injectivity radius bound of $\tilde U_{i}$ at $q$ because $inj_{\tilde U_{i}}(\bar S_{i}) \geq inj_{U_{i}}(S_{i}) > \epsilon_{0}$.  Finally since $\tilde p_{i}\rightarrow \tilde p_{\infty}$ we see that for $i$ sufficiently large that $\tilde U_{i}$ satisfies \ref{thm_main_collapse} $(1)$ with $\frac{\epsilon_{0}}{2}>\frac{1}{i}$.

Now we show the structure of limit and estimate the isotropy group.  After passing to a subsequence we can pick $U_{i}$ as above.  Let $q_{i}\in S_{i}$ and notice that $d(p_{i},q_{i})\rightarrow 0$.  It follows from the lower injectivity radius bound  and because $g_{i}$ are $\{A\}_{0}^{k}$-regular that we can now pass to a subsequence such that if $\bar p_{i},\bar q_{i}\in \tilde U_{i}$ are lifts of $p_{i}$ and $q_{i}$ with $d(\bar p_{i},\bar q_{i})=d(p_{i},q_{i})$ then $(\tilde U_{i},g_{i}, \Lambda_{i}, \bar q_{i})\stackrel{eGH - C^{k+1,\alpha}}{\rightarrow}(\tilde U_{\infty}, g_{\infty}, N, \bar q_{\infty})$, where $N$ is a closed subgroup of the isometry group of $g_{\infty}$.  Because there is a subgroup of $\Lambda_{i}$ (namely $\tilde\Lambda_{i} \equiv \Lambda_{i}\cap N_{i}$) of bounded index which is at most $n$-step nilpotent, it follows that there is a subgroup of $N$ of bounded index which is at most $n$-step nilpotent.  In particular the identity component of $N$ is nilpotent.  Now note that if $dim(S_{i})=0$ then $S_{i}$ is discrete and thus because of the normal injectivity radius bound on $S_{i}$ this implies a uniform lower injectivity radius bound on $M_{i}$.  Hence if we assume that $inj_{M_{i}}(p_{i})\rightarrow 0$ then $dim(\bar S_{i})=dim(S_{i})\geq 1$.  Notice also that $\bar S_{i}$ is invariant under the action of the fundamental group $\Lambda_{i}$, and further because the diameter of $S_{i}$ is tending to zero the action of $\Lambda_{i}$ on $\bar S_{i}$ maps $\bar q_{i}$ to increasingly dense subsets of $\bar S_{i}$.  The second fundamental form bounds on $\bar S_{i}$ and the uniform lower bound on the injectivity radius of $\bar S_{i}$ guarantee that $\bar S_{i}$ limits, at least on compact subsets of $\tilde U_{\infty}$, to a submanifold $\bar S_{\infty}\subseteq \tilde U_{\infty}$ which also satisfies $dim(\bar S_{\infty})\geq 1$.  But $N$ now acts transitively on $\bar S_{\infty}$, in particular $\bar S_{\infty} = N\cdot \bar q_{\infty} = N\cdot \bar p_{\infty}$ and because $N$ is closed in the isometry group we can identify $\bar S_{\infty} = N/I_{\bar p_{\infty}}$. Since $dim(\bar S_{\infty})>0$ we have $dim N >dim I_{\tilde p_{\infty}}$.  Further because $\bar S_{\infty}$ is connected, since $\bar S_{i}$ are, and $I_{\bar p_{\infty}}$ is compact we see that $N$ is at most a finite extension of its nilpotent identity component.

\end{proof}

\begin{remark}
In the above if $f_{i}$ were not only almost Riemannian submersions in the $C^{1}$ sense but in the $C^{2}$ sense (which is not guarenteed by \cite{CFG}) then we could see that not only is the normal injectivity radius of $S_{i}$ bounded from below, but as $i$ grows should be approaching the injectivity radius of $\tilde S_{\infty}$ in $Y_{\infty}$.  This will be a consequence of some estimates in part II of this paper.
\end{remark}

\section{Geometry of Toric Quotients}\label{sec_gtq}

Many of the techniques of this section are valid for arbitrary quotient geometries, however we will only derive sharp estimates for quotients by finite extensions of tori.  The main purpose of this section is to prove the following

\begin{theorem}\label{thm_main_torus}
Let $(M^{n},g)$ be an $n$-dimensional Riemannian Manifold with $C^{k,\alpha}$ metric and $G$ be an $l$ dimensional Lie Group with a faithful, proper, and isometric action on $M$. Assume for each $x\in M$ that the isotropy is a finite extension of a torus $I_{x} = \tilde T$ with $dim I_{x}\leq i$.  Let $B\subseteq M/G$ be the set of non-$C^{k,\alpha}$ orbifold points, then $dim_{Haus}B \leq min\{n-(l-i)-4,dim M/G-3\}$.
\end{theorem}

We spend much of this section proving a Euclidean version of the above theorem:

\begin{prop}\label{prop_m_torus}
Let $(\mathds{R}^{m},g)$ be a $C^{k,\alpha}$ Riemannian manifold and assume $\widetilde{T}\leq O(m)$ is a finite extension of a torus which acts faithfully and isometrically on $(\mathds{R}^{m},g)$.  Then if $B\subseteq (\mathds{R}^{m},g)/\widetilde{T}$ is the set of non-$C^{k,\alpha}$ orbifold points, then $dim_{Haus}B \leq min\{m-4,dim \mathds{R}^{m}/\widetilde{T}-3\}$.
\end{prop}

The proof of the proposition will be done in several parts.  In Lemma \ref{lem_tor_3} we first prove it in the case when $m\leq 3$ and $\tilde T=T$ is a torus, so that $B$ is empty in this case.  With some of the tools of this section we will then be able to prove it for larger $m$ inductively, and by using the results of Appendix \ref{sec_orb} we will show the result still holds when we take a finite extension of $T$.

The next simple lemma will be useful when we move to studying stratified spaces.

\begin{lemma}
Let $(M,g)$, $(N,k)$ be Riemannian manifolds and $\pi:M\rightarrow N$ a Riemannian
submersion.  Define the semi-definite tensor $h$ on $M$ by $h(v,w)=g(p_{\mathcal{H}}(v),p_{\mathcal{H}}(w))$,
where $p_{\mathcal{H}}$ is the projection to the horizontal distribution.  Let $S\subseteq M$ be a submanifold
such that $\pi|_{S}$ is a diffeomorphism onto its image.  Then $h|_{S}$ is a Riemannian metric and
$\pi|_{S}:(S,h|_{S})\rightarrow \pi(S)$ is an isometry.
\end{lemma}
\begin{proof}
Let $v,w\in T_{x}S$ and $\mathcal{V}$, $\mathcal{H}$ be the vertical and horizontal distributions associated to the Riemannian submersion, respectively.  Then $h(v,v)=0$ $\Leftrightarrow$ $p_{\mathcal{H}}(v)=0$ $\Leftrightarrow$ $v\in\mathcal{V}$ $\Leftrightarrow$ $\pi_{\ast}(v)=0$ $\Leftrightarrow$ $v=0$ since $\pi|_{S}$ is a diffeomorphism onto its image.  Now let $v^{*},w^{*}\in\mathcal{H}$ such that $\pi_{*}(v^{*})=\pi_{*}(v)$ and $\pi_{*}(w^{*})=\pi_{*}(w)$, hence $w^{*}=p_{\mathcal{H}}(w)$ and $v^{*}=p_{\mathcal{H}}(v)$.  Then $k(\pi_{*}(v),\pi_{*}(w)) = k(\pi_{*}(v^{*}),\pi_{*}(w^{*})) = g(v^{*},w^{*})=h(v,w)$ as claimed.
\end{proof}

The purpose of the above is simply to see an intrisic way to write the pullback of the metric on $N$ to any submanifold $S$.  This intrinsic viewpoint will generalize in a useful way in the stratified category.

Recall now that if $M$ is a smooth manifold and $G$ is a Lie Group acting smoothly and properly on $M$, then there is a natural stratification structure on $M$ given by collecting together points in $M$ whose isotropy groups lie in the same conjugacy class.  This in turn induces a stratification structure on the quotient space $M/G$.  If $M$ is Riemannian and $G$ acts isometrically, then the action induces both a quotient distance function and a stratified Riemannian structure on the quotient space, see \cite{Pf}.  A stratified Riemannian structure on a stratified space always itself induces a length distance function with respect to piecewise stratified curves (generally this means continuous curves which can be decomposed into a countable union of curves, each of which lie in a single stratum and are smooth.  The length of such curves is induced from the stratified Riemannian structure.  In fact, for a quotient space, we can get away with curves which can be decomposed into a finite number of such pieces). It can be checked that the induced length space distance and the quotient distance function are the same on $M/G$.

\begin{definition}
Let $(M,g)$ be a Riemannian manifold with $G$ a Lie group acting properly and isometrically on it.  Let $p_{\mathcal{H}}$ be the projection to the horizontal distribution at each point.  We call the tensor $h$ on $M$ defined by $h(v,w)=g(p_{\mathcal{H}}(v),p_{\mathcal{H}}(w))$ the full pullback.
\end{definition}

Note that $h$ above is not the pullback of the stratified Riemannian metric on $M/G$.  It is in fact larger than the standard pullback.  Additionally it is worth pointing out that $h$ is not even continuous, as $p_{\mathcal{H}}$ isn't continuous away from the principal stratum.

\begin{definition}
Let $M$ be a smooth manifold with $G$ a Lie Group acting smoothly and properly on $M$. Let $S\subseteq M$ be a smooth submanifold and $\Gamma \leq G$ a finite subgroup.  We say the pair $(S,\Gamma)$ is a local reduction of the group action if

1) $\Gamma$ restricts to an action on $S$

2) Giving $S$ the induced stratification structure from $M$ and letting $\iota: S/\Gamma \rightarrow M/G$ be the natural map with $s\Gamma \mapsto s G$, then the image of $\iota$ is open and $\iota$ is a stratified diffeomorphism onto its image.

We say $(S,\Gamma)$ is a local reduction at $x$ if $S$ contains $x$.
\end{definition}

\begin{prop}\label{prop_lr}
Let $(M,g)$ be a Riemannian Manifold with $G$ a Lie Group acting properly and isometrically on $M$.  Let $(S,\Gamma)$ be a local reduction.  Then the full pullback $h$ defines a stratified metric on $S/\Gamma$, hence induces a distance function on $S/\Gamma$, and $\iota: S/\Gamma \rightarrow \iota(S/\Gamma)\subseteq M/G$ is an isometry of metric spaces with respect to this distance.
\end{prop}
\begin{proof}
Let $\mathcal{S}\subseteq M$ be a stratum of $M$ which intersects $S$.  Let $\pi_{G}:M\rightarrow M/G$ and $\pi_{\Gamma}:S \rightarrow S/\Gamma$ be the projection maps.  Note that $\pi_{G}|_{\mathcal{S}}$ is a Riemannian submersion and $\pi_{\Gamma}|_{\mathcal{S}}$ is a covering map.  Let $p\in (\mathcal{S}\cap S)/\Gamma$ and $U$ a neighborhood of $p$ in the stratum $(\mathcal{S}\cap S)/\Gamma$.  Since the projection is a covering map we can lift $l:U\rightarrow\tilde{U}\subseteq\mathcal{S}$ in $M$, for $U$ small and $l$ a diffeomorphism.  Now for $x\in U$ we have that $\iota(x)=\pi_{G}(l(x))$.  Since by assumption the restriction of $\iota$ to $U$ becomes a diffeomorphism onto its image we have that the same holds for $\pi_{G}$ restricted to $\tilde{U}$, so by the last lemma that $h|_{\tilde{U}}$ is a metric on $\tilde U$ and $\pi_{G}|_{\tilde{U}}$ is a Riemannian isometry onto its image.  Since $h$ is invariant under the $\Gamma$ action and $p$ was arbitrary we see that $h$ induces Riemannian structure on $(\mathcal{S}\cap S)/\Gamma$ and the restriction of $\iota$ to this stratum is a Riemannian isometry onto its image.  Since the stratum $\mathcal{S}$ was arbitrary we see that $h$ induces a stratified Riemannian structure on $S/\Gamma$ and $\iota$ is a stratified Riemannian isometry onto its image.  Finally we point out as before that the quotient distance function on $M/G$ is induced by piecewise stratified curves, and hence a stratified Riemannian isometry induces an isometry of metric spaces with induced length space structures.
\end{proof}

\begin{definition}
Let $M$ be a smooth manifold with $G$ a Lie Group acting smoothly and properly on $M$. Let $S\subseteq M$ be a smooth submanifold and $\Gamma \leq G$ a finite subgroup.  We say the pair $(S,\Gamma)$ is a local Riemannian reduction of the group action if $(S,\Gamma)$ is a local reduction and additionally it holds that for every $C^{k,\alpha}$ metric $g$ on $M$ that is invariant under the $G$ action, the restriction of the full pullback tensor $h$ to $S$ is a $C^{k,\alpha}$ metric.
\end{definition}

\begin{remark}
It can be checked that this condition is equivalent to having that the restrictions of either of the projection maps $p_{\mathcal{H}}$ or $p_{\mathcal{V}}$ to maps $TS\rightarrow TM$ are $C^{k,\alpha}$ maps.  Also it may not seem like it but we will see that this condition really only depends on the structure of the group action and not on any underlying geometry.
\end{remark}

\begin{corollary}\label{cor_lRr}
Let $(M,g)$ be a $C^{k,\alpha}$ Riemannian Manifold with $G$ a Lie Group acting properly and isometrically on $M$.  Let $(S,\Gamma)$ be a local Riemannian reduction, then $\pi(S)\subseteq M/G$ has a $C^{k,\alpha}$ Riemannian Orbifold structure.
\end{corollary}

\begin{proof}
The length space distance function on $S/\Gamma$ is then the same as the quotient distance function induced by $(S,h)/\Gamma$ and by the last lemma $\iota: S/\Gamma \rightarrow \pi(S)$ is an isometry of metric spaces.
\end{proof}

Notice in the case of a finite group action a local reduction and a local Riemannian reduction are necessarily the same, this need not be the case when the group has dimension.  Also for a proper Lie Group action by $G$ on a smooth manifold $M$ there exists a trivial example of a local Riemannian reduction at a point $x\in M$ when the isotropy at that point is finite, in particular let $\Gamma = I_{x}$ and $S=S_{x}$ be a slice through $x$ (rigorously this is a corollary of Lemma \ref{lRr_isotropy}, though it is at least intuitively clear).  It may not be immediately clear that when the isotropy is not finite that a local Riemannian reduction will even exist, but the next lemma will show that there are in fact many such examples.

\begin{lemma}\label{lem_tor_3}
Let $(\mathds{R}^{m},g)$ be a $C^{k,\alpha}$ Riemannian manifold with $m\leq 3$.  Assume $T\leq O(m)$ is a torus which acts isometrically on $(\mathds{R}^{m},g)$.  Then $(\mathds{R}^{m},g)/T$ has a $C^{k,\alpha}$ Riemannian Orbifold structure.
\end{lemma}
\begin{proof}
$m=1$:  There is only the trivial action.

$m=3$:  After possibly conjugating coordinates there is only one torus $O(3)$ action, which is the action of $T=S^{1}$ by rotation around the $z$-axis.  Let $S=xz-plane$ and $\Gamma=\{1,e^{i\pi}\}$. Note that the orbit map $\iota:S/\Gamma \rightarrow \mathds{R}^{3}/T$ is a homeomorphism and its restriction to the two strata of $\mathds{R}^{3}$ and $S$, namely the $z$-axis and $\mathds{R}^{3}-\{z$-axis$\}$, are diffeomorphisms.  Hence $(S,\Gamma)$ is a local reduction of the group action.  We show $(S,\Gamma)$ is a local Riemannian reduction and then can apply the last corollary.  So $h|_{S}$ is $C^{k,\alpha}$ iff $p_{\mathcal{H}}|_{S}$ is $C^{k,\alpha}$ iff $p_{\mathcal{V}}|_{S}$ is $C^{k,\alpha}$.  Let $\mathcal{Y}$ be the smooth distribution which at each point is just the $y$-axis.  Then clearly $p_{\mathcal{Y}}$ is $C^{k,\alpha}$ globally.  Let $(x,0,z)\in S$ with $x\neq 0$,  then $\mathcal{V}=\mathcal{Y}$ and hence $p_{\mathcal{V}}=p_{\mathcal{Y}}$.  If $x=0$ then because $g$ is invariant by rotation around the $z$-axis we have that $p_{\mathcal{Y}}(v)=0$ $\forall v\in T_{(0,0,z)}S$ and that $p_{\mathcal{V}}(v)=0$
for all $v$.  Hence $p_{\mathcal{V}}|_{S} = p_{\mathcal{Y}}|_{S}$ and so $p_{\mathcal{V}}|_{S}$ is $C^{k,\alpha}$.

$m=2$:  Same as above but use $S$=$x$-axis (alternatively notice that in this case quotient is always just the isometric half line, which has the natural orbifold cover by the real line).
\end{proof}
\begin{remark}
In fact in the above we could also have assumed $T$ is a finite abelian extension of a torus without any change in the proof, since if $T$ were finite then we could take the whole of $\mathds{R}^{m}$ as our reduction, while if $T$ was not finite then it must still only be the circle.
\end{remark}

Now given Proposition \ref{finite_ext_quot} we have an immediate corollary.

\begin{corollary}\label{cor_3_torus}
Let $(\mathds{R}^{m},g)$ be a $C^{k,\alpha}$ Riemannian manifold, $m\leq 3$.  Assume $\widetilde{T}\leq O(m)$ is a finite extension of a torus which acts isometrically on $(\mathds{R}^{m},g)$.  Then $(\mathds{R}^{m},g)/\widetilde{T}$ has a $C^{k,\alpha}$ Riemannian Orbifold structure.
\end{corollary}

To finish Proposition \ref{prop_m_torus} (and Theorem \ref{thm_main_torus}) we need the following two lemmas for the inductive procedure.

\begin{lemma}\label{lRr_product}
Let $M$ be a smooth manifold with Lie Group $G$ acting smoothly and properly on it.   Let $(S,\Gamma)$ be a local Riemannian reduction in $M$.  Then for the $G$ action on $M\times\mathds{R}$, where $G$ acts by acting fixing the $\mathds{R}$ factor and acting on each $M$ fiber, we have that $(S\times\mathds{R},\Gamma)$ is a local Riemannian reduction in $M\times\mathds{R}$.
\end{lemma}
\begin{proof}
First note that $(S\times\mathds{R},\Gamma)$ is clearly a local reduction.  As in the proof of proposition \ref{prop_lr} we observe that $(S\times\mathds{R},\Gamma)$ is a local Riemannian reduction iff for each $C^{k,\alpha}$ metric $g$ on $M$ which is invariant under $G$ that the restriction of $p_{\mathcal{\tilde{V}}}$ to $S\times\mathds{R}$ is $C^{k,\alpha}$, where $\tilde{\mathcal{V}}$ is the vertical distribution with respect to the $G$ action on $M\times\mathds{R}$.  But since $G$ acts trivially on each $\mathds{R}$ factor, we see that $p_{\mathcal{\tilde{V}}} = p_{\mathcal{V}}\circ p_{M} $ where $p_{M}$ is the projection to the tangent of $M$ and $p_{\mathcal{V}}$ is the projection from $M$ to the vertical distribution induced by the action of $G$ on $M$.  Thus since $p_{M}$ is $C^{k,\alpha}$ and $p_{\mathcal{V}}$ restricted to $S$ is $C^{k,\alpha}$ we have that $p_{\mathcal{\tilde{V}}}$ restricted to $S\times\mathds{R}$ is $C^{k,\alpha}$.
\end{proof}

The following requires some of the structure from Appendix \ref{sec_gq} to prove.  Note that if $G$ is a smooth and proper action and $x\in M$, then there is an $O(m)$ action of $I_{x}$ on $\mathds{R}^{m}$ defined by a choice of slice $(S_{x},\phi_{x})$ at $x$ (and it is easy to check this action is, up to conjugation, independent of the slice map $\phi_{x}$, see Appendix \ref{sec_gq}).

\begin{lemma}\label{lRr_isotropy}
Let $M$ be a smooth manifold with $G$ a Lie Group acting smoothly and properly on $M$, and  let $x\in M$.  If there exists a local Riemannian reduction for the induced action of $I_{x}$ on $\mathds{R}^{m}$ at $0$ then there exists a local Riemannian reduction for the action of $G$ on $M$ at $x$.
\end{lemma}

\begin{proof}
Let $(S_{x},\phi)$ be a slice at $x$ and $(S^{m},\Gamma)$ be a local Riemannian reduction at $0\in \mathds{R}^{m}$ with respect to the induced $I$ action.  Define $S=\phi^{-1}(S^{m})\subseteq S_{x}$.  It is clear that $(S,\Gamma)$ is a local reduction for $G$ at $x$, we need to prove it is a Riemannian reduction.  Let $\mathcal{V}$ be the (noncontinuous) vertical distribution on $G_{S}$ and $g$ a $C^{k,\alpha}$ $G$-invariant metric on $G_{S}$ ($G_{S}\equiv S_{x}\cdot G$, see Appendix \ref{sec_gq}).  We need to show the projection $p_{\mathcal{V}}|_{S}$ is $C^{k,\alpha}$.

Let $\mathcal{I}\subseteq \mathcal{V}$ be the vertical distribution generated by the $I\leq G$ action, and let $\mathcal{I}^{\perp}\subseteq\mathcal{V}$ be the perpendicular of $\mathcal{I}$ in $\mathcal{V}$ with respect to $g$.  Then $p_{\mathcal{V}} = p_{\mathcal{I}} + p_{\mathcal{I}^{\perp}}$ and if we can show $p_{\mathcal{I}}|_{S}$ and $p_{\mathcal{I}^{\perp}}|_{S}$ are $C^{k,\alpha}$ then $p_{\mathcal{V}}|_{S}$ is also $C^{k,\alpha}$.  Since $\mathcal{I}$ is tangent to $S_{x}$ we see that $p_{\mathcal{I}}|_{S} = \phi_{*}^{-1}\circ p_{\mathcal{I}^{m}}\circ\phi_{*}|_{S}$ where $\mathcal{I}^{m}$ is the horizontal distribution on $\mathds{R}^{m}$ generated by $I$.  Since this is a composition of $C^{k,\alpha}$ functions it is also $C^{k,\alpha}$.

We will show that $\mathcal{I}^{\perp}$ is in fact a $C^{k,\alpha}$ distribution, hence the projection map to it is as well.  So let $\mathds{R}^{m}\times G$ be the total space above $G_{S}$ with $\tilde g$ a total metric as in Appendix \ref{sec_gq}.  Also as in Appendix \ref{sec_gq} consider the trivial bundle $G_{S}\times\mathfrak{g}$ equipted with a $C^{k,\alpha}$ adjoint metric $\tilde h$ and let $e_{*}:G_{S}\times\mathfrak{g}\rightarrow TG_{S}$ be the vector bundle map by pushing forward the lie algebra by the derivative of the $G$ action.  The construction of $\tilde h$ was such that for any $p\in S_{x}$ if $i_{p}$ is the lie algebra of the isotropy $I_{p}$ and $i_{p}^{\perp}$ is its perpendicular with respect to $\tilde h$, then $e_{*}(p,\cdot)$ restricted to $i_{p}^{\perp}$ is an isometry.  In particular since $\forall p\in S_{x}$ we have $i_{p}\leq i_{x}$ we have that $e_{*}(p,\cdot)$ restricted to $i_{x}^{\perp}$ is an isometry for all $p$ and $\mathcal{I}^{\perp}(p)=e_{*}(p,i_{x}^{\perp})$.  But $i_{x}$ is a constant distribution in $G_{S}\times\mathfrak{g}$ and $\tilde h$ is $C^{k,\alpha}$, so $i_{x}^{\perp}$ is a $C^{k,\alpha}$ distribution in $G_{S}\times\mathfrak{g}$.  Since $e_{*}$ is an isometry on this distribution we thus have $e_{*}(i_{x}^{\perp}) =\mathcal{I}^{\perp}$ is a $C^{k,\alpha}$ distribution, as claimed.

\end{proof}
\begin{remark}\label{rem_lRr_product}
Using techniques similar to the last two lemmas one can also show that if $M$ and $N$ are smooth with $G$ acting smoothly and properly on $M$ and $H$ on $N$, then if $(m,n)\in M\times N$ is such that there exists a local Riemannian reduction for the $G$ action at $m\in M$ and for the $H$ action at $n\in N$, then there exists a local Riemannian reduction for the $G\times H$ action at $(m,n)\in M\times N$.
\end{remark}

The difficulty in the above is that both $\mathcal{V}$ and $\mathcal{I}$ are not even continuous, yet their difference $\mathcal{I}^{\perp}$ has full regularity.  The pullback metric on $G_{S}\times\mathfrak{g}$ from the map $e_{*}$ is degenerate, and so trying to construct the perpendicular $\mathcal{I}^{\perp}$ directly in $\mathfrak{g}$ you find it is not well defined. The process of constructing a smoothing of this pullback metric, namely the adjoint metric $\tilde h$, allows you to do this construction in a canonical and hence smooth way.

Now we can finish the proof of Proposition \ref{prop_m_torus}:

\begin{proof}[Proof of Proposition \ref{prop_m_torus}]

We need only assume $\tilde T = T$ is just a torus, as then Proposition \ref{finite_ext_quot} shows the result holds for $\tilde T$, a finite extension of $T$.  In the following it will be convenient to allow a little more however, that $\tilde T = T$ is at most a finite abelian extension of a torus.

We will first prove by induction on $m$ that if $\tilde B\subseteq\mathds{R}^{m}$ is the set of points which do not have local Riemannian reductions with respect to the  $T$ action, then $dim_{Haus}\tilde B\leq m-4$.  To begin the induction argument note that for $m\leq 3$ that the proof of Lemma \ref{lem_tor_3} shows that the result holds.  Now assume it holds for some $m\geq 3$, we will show it holds for $m+1$.  So let $(\mathds{R}^{m+1},g)$ be a $C^{k,\alpha}$ Riemannian manifold with $T$ a finite abelian extension of a torus with an $O(m+1)$ isometric action.  Then we can write $\mathds{R}^{m+1}-\{0\}=\mathds{R}\times S^{m}$ with $T$ acting on each $S^{m}$ fiber.  If $T$ acts on $S^{m}$ then the dimension of a slice through any point certainly has dimension $\leq m$, and since the isotropy of this action must also be a finite abelian extension of a torus we have by the inductive hypothesis and Lemma \ref{lRr_isotropy} that the subset $\tilde B_{S}\subseteq S^{m}$ of points without local Riemannian reductions satisfies $dim_{Haus}\tilde B_{S}\leq m-4$.  Thus by Lemma \ref{lRr_product} for every $(t,x)$ outside set $\mathds{R}\times \tilde B_{S}$ there exists a local Riemannian reduction and $dim_{Haus}\mathds{R}\times \tilde B_{S}\leq m-3 = (m+1)-4$ as claimed.

Now notice that $\tilde B$ is $T$ invariant and if $B\subseteq (\mathds{R}^{m},g)/T$ is the set of non-$C^{k,\alpha}$ orbifold points then $B\subseteq \tilde B/T$ with $dim \tilde B/T \leq dim \tilde B \leq m-4$, which gives the first inequality.

To obtain the second inequality we need to estimate more carefully the dimension of $\tilde B/T$.  So let $\mathds{R}^{m}=V^{1}\oplus \ldots \oplus V^{n}$ (if $m=2n$ is even dimensional, $\mathds{R}^{m}=V^{1}\oplus \ldots \oplus V^{n}\oplus\mathds{R}$ if $m=2n+1$ is odd dimensional) be the decomposition of $\mathds{R}^{m}$ into two dimensional orthogonal subspaces determined by the action of a maximal torus in $O(m)$ which contains $T$.  Let $F\leq T$ be the subgroup of $T$ which fixes $\tilde B$ with $F_{0}$ its identity component.  Because the actions of $T$ and $F_{0}$ are generically free we see that $dim (\mathds{R}^{m}/T) = m-dim T$ and $dim \tilde B/T = dim\tilde B - dim T + dim F_{0}$. Because $F_{0}$ is also a torus action there is some subcollection of $\{V^{i}\}$ that $F_{0}$ acts nontrivially on while it fixes the rest.  After reordering we can assume $F_{0}$ acts nontrivially on $V^{1},\ldots, V^{k}$ while it acts trivially on $V^{k+1},\ldots, V^{n}$.  Because $F_{0}$ is a torus action we then see that the fix point set $Fix(F_{0})=\{0\}\times V^{k+1}\oplus\ldots\oplus V^{n}$ (since if the projection of a point in $\mathds{R}^{m}$ into $V^{i}$ is nonzero for some $i\in\{1,\ldots, k\}$ then some element of $F_{0}$ rotates the point).  Now on the one hand since $F_{0}$ embeds as a torus action on $V^{1}\oplus\ldots\oplus V^{k}$ we have the estimate $dim F_{0} \leq k$.  Combining with the estimate $dim\tilde B\leq m-4$ we see

\[
dim \tilde B/T \leq m - dim T + k -4 = dim (\mathds{R}^{m}/T) + k -4.
\]

But also $\tilde B \subseteq Fix(F_{0})=\{0\}\times V^{k+1}\oplus\ldots\oplus V^{n}$ by construction and so we can also estimate $dim\tilde B \leq m -2k$ to get

\[
dim \tilde B/T \leq m - 2k - dim T  + k = dim (\mathds{R}^{m}/T) - k.
\]

If $k=0,1$ then the first estimate gives us $dim \tilde B/T \leq dim (\mathds{R}^{m}/T) - 3$, while if $k\geq 3$ then the second estimate also gives us $dim \tilde B/T \leq dim (\mathds{R}^{m}/T) - 3$.  We thus only need to analyze the $k=2$ case.  If $dim F_{0}=1$ then we are done since then $dim\tilde B \leq dim (\mathds{R}^{m}/T) - 3$, so we need only consider $dim F_{0}=2$.  In this case $F_{0}$ embeds as the maximal torus in $O(V^{1}\oplus V^{2})$ and so by Lemma \ref{lem_tor_split} the quotient $V^{1}\oplus V^{2}/F_{0}$ is an orbifold and $V^{1}\oplus V^{2}$ admits a local Riemannian reduction at every point.  Hence by remark \ref{rem_lRr_product} if we consider the action of $T/F_{0}$ on $V^{k+1}\oplus\ldots\oplus V^{n}$ then any point which admits a local Riemannian reduction with respect to this action admits one in $\mathds{R}^{m}$ for the $T$ action, and so by the previous estimates $dim\tilde B \leq dim(V^{k+1}\oplus\ldots\oplus V^{n})-4 = m-8$ and hence $dim\tilde B/T \leq m - 6$.

\end{proof}

Finally we can prove the main theorem of this section

\begin{proof}[Proof of Theorem \ref{thm_main_torus}]
Let $x\in M$ and $I_{x}$ be the isotropy group with $dim I_{x} \leq i$ and $(S_{x},\phi_{x})$ a smooth slice through $x$ as in Appedix \ref{sec_gq}.  Then $I_{x}$ has an $O(m)$ action on $\mathds{R}^{m}$ where $m=n-(dim G - dim I_{x})\leq n - (l-i)$ is the dimension of a slice through $x$. By Corollary \ref{cor_geom_slice} there exists a $C^{k,\alpha}$ metric $h$ on $\mathds{R}^{m}$ and an $O(m)$ action by $I_{x}$ which is isometric with respect to $h$ such that a neighborhood of $x G \in M/G$ is isometric to $(\mathds{R}^{m},h)/I_{x}$.  Thus by Proposition \ref{prop_m_torus} the set $B\subseteq (\mathds{R}^{m},h)/I_{x}$ of non-$C^{k,\alpha}$ orbifold points satisfies $dim B \leq \min\{n-(l-i)-4,m-i-3\} = min\{n-(k-i)-4, dim M/G -3\}$.  Since $x$ was arbitrary this completes the proof.
\end{proof}

\section{Geometry of Toric Quotients II}\label{sec_gtq2}
In this section we will give a necessary and sufficient condition for a torus action to have a local Riemannian reduction at a point.  It is interesting to note that the results of this section can in fact be generalized to nontorus actions, however the techniques for this are a little different and we will not need this result.  For notational convenience we introduce the following definitions

\begin{definition}
Let $T^{k}$ be a torus acting orthogonally on $\mathds{R}^{m}$.  We call the action by $T^{k}$ split if for every two dimensional subspace $V\subseteq\mathds{R}^{m}$ which is invariant under the $T^{k}$ action but not fixed, then there exists $S^{1}\leq T^{k}$ such that the induced $S^{1}$ action rotates $V$ but fixes its orthogonal compliment.
\end{definition}

\begin{lemma}\label{lem_split_curv}
Let $T^{k}$ be a torus which acts orthogonally on $\mathds{R}^{m}$ and let $\{V^{T}_{i}\}$ be its orthogonal decomposition.  Then the quotient $\mathds{R}^{m}/T^{k}$ has bounded stratified curvature if and only if the action is split.
\end{lemma}

\begin{proof}
We can identify $T^{k}$ with its image in $O(m)$ without any loss of generality.  First assume the action is not split, then we can find a maximal torus $T^{n}$ containing $T^{k}$ such that when we write $\mathds{R}^{m}$ in coordinates $(x_{1},y_{1}, \ldots , x_{n},y_{n})$ where $V^{T}_{i}=span\{x_{i},y_{i}\}$ (if $m=2n$, otherwise we use coordinates $(x_{1},y_{1}, \ldots , x_{n},y_{n},z_{n+1})$ if $m=2n+1$, though since $z_{n+1}$ is fixed there is no loss of generality in assuming $m=2n$) form an orthogonal decomposition with respect to the action by $T^{n}$ and $V^{T}_{1}$ is a two dimensional subspace for which $T^{k}$ fails the split assumption.  Let $U\subseteq\mathds{R}^{m}$ be the open dense subset defined by $U\equiv \{(x_{1},y_{1},\ldots,x_{n},y_{n})\in\mathds{R}^{m}: x_{i},y_{i}\neq 0$ $\forall i\}$.  Notice that $T^{n}$ acts freely on $U$.  Assuming the action is not split we will construct a point $\bar{x}\in U$ and orthogonal horizontal vectors $v,w\in \mathcal{H}_{T^{k}}$ at $\bar{x}$ such that the vertical projection $p_{\mathcal{V}_{T^{k}}}([v,w])\neq 0$ (this is a tensor, and so independent of how you extend $v$ and $w$ to compute the lie bracket).  Thus at $\tilde{x}=\bar{x}T^{k}\in \mathds{R}^{m}/T^{k}$, which is a smooth point of the quotient, the sectional curvatures strictly increase.  Since the quotient is a cone this forces the sectional curvatures along the line connecting the origin to $\tilde{x}$ to blow up as you approach the origin.

Now since $T^{k}$ fails the split assumption for $V^{T}_{1}$ there is no $S^{1}\leq T^{k}$ which rotates only $V^{T}_{1}$, however $V^{T}_{1}$ is also not a fixed plane.  In particular, the killing field $\bar{v}=(-y_{1},x_{1},0, \ldots , 0)$ generated from an $S^{1}$ action which rotates only $V^{T}_{1}$ is not contained in $\mathcal{V}_{T^{k}}$ at any point where it is nondegenerate, in particular at any point in $U$.  Let $\bar{w}=(-x_{2},0,x_{1},0,\ldots,0)$ and fix for the moment $\bar{x}\in U$.  Let $\bar{v}_{T^{k}}$ be the killing field induced by an element of the lie algebra of $T^{k}$ with the property that at $\bar{x}$ we have that $\bar{v}_{T^{k}}$ is the projection of $\bar{v}$ into the vertical subspace $\mathcal{V}_{T^{k}}$.  Similarly let $\bar{w}_{T^{n}}$ be the killing field induced by an element of the lie algebra of $T^{n}$ such that at $\bar{x}$ we have that $\bar{w}_{T^{n}}$ is the projection of $\bar{w}$ into the vertical subspace $\mathcal{V}_{T^{n}}$.  Let $v=\bar{v} -\bar{v}_{T^{k}}$ and $w=\bar{w}-\bar{w}_{T^{n}}$, which at $\bar{x}$ are horizontal projections of $\bar{v}$ and $\bar{w}$.  We show that at $\bar{x}$ we have that $v$ and $\bar{v}_{T^{k}}$ are nonzero.  Since $V^{T}_{1}$ is not a fixed plane and any element of $T^{k}$ which moves $V^{T}_{1}$ must move some other $V^{T}_{i}$ as well, we have that any element of $\mathcal{V}_{T^{k}}$ with a component in the $V^{T}_{1}$ direction must have a component in the $V^{T}_{i}$ direction.  Since $\bar{v}\not\in\mathcal{V}_{T^{k}}$ by construction, we have that $\bar{v}_{T^{k}}$ is nonzero with a component in the $V^{T}_{1}$ and hence $V^{T}_{i}$ directions.  Since $\bar{v}$ does not have a component in the $V^{T}_{i}$ direction the difference $v$ is nonzero.

Now by construction there exists $\lambda_{i}=\lambda_{i}(\bar{x})$ and $\mu_{i}=\mu_{i}(\bar{x})$ such that $\bar{v}_{T^{k}} = (-\lambda_{1}y_{1}, \lambda_{1}x_{1},-\lambda_{2}y_{2}, \lambda_{2}x_{2},\ldots)$ and $\bar{w}_{T^{n}} = (-\mu_{1}y_{1}, \mu_{1}x_{1},-\mu_{2}y_{2}, \mu_{2}x_{2},\ldots)$.  In particular we have that $w = (x_{2}+\mu_{1}y_{1}, -\mu_{1}x_{1},x_{1}+\mu_{2}y_{2}, -\mu_{2}x_{2},\ldots)$.  At $\bar{x}$ we see that if $w=0$ then $\mu_{1} = \frac{\bar{x}_{2}}{\bar{y}_{1}} \neq 0 \Rightarrow -\mu_{1}\bar{x}_{1} \neq 0 $, which is not possible and hence $w\neq 0$.  Also note that by construction that $\langle v,w\rangle = 0$ because $v\in \mathcal{V}_{T^{n}}$ but $w\in \mathcal{H}_{T^{n}}$.

We know by our construction that $\bar{v}\not\in \mathcal{V}_{T^{k}}$ or $\mathcal{H}_{T^{k}}$ at $\bar{x}$, and hence $\bar{v}\cdot\bar{v}_{T_{k}}>0$ and $|\bar{v}|^{2}>|\bar{v}_{T_{k}}|^{2}$ at $\bar{x}$.  Hence $\lambda_{1}(\bar{x}_{1}^{2}+\bar{y}_{1}^{2})>0 \Rightarrow \lambda_{1}>0$ and $(\bar{x}_{1}^{2}+\bar{y}_{1}^{2}) > \lambda^{2}_{1}(\bar{x}_{1}^{2}+\bar{y}_{1}^{2}) \Rightarrow \lambda_{1} < 1$.  Now $[v,w] = [\bar{v} - \bar{v}_{T_{k}}, \bar{w} - \bar{w}_{T_{n}}] = [\bar{v} - \bar{v}_{T_{k}}, \bar{w}]$ (since $v$ and $\bar{w}_{T_{n}}$ are both pushforwards from a commuting lie algebra) $= (-\lambda_{2}y_{2},(1-\lambda_{1})x_{2},(-1+\lambda_{1})y_{1},\lambda_{2}x_{1},0,\ldots,0)$. Hence at $\bar{x}$ we have that $\langle[v,w],\bar{v}_{T_{k}}\rangle = \bar{y}_{1}\bar{y}_{2}(\lambda_{1}\lambda_{2}+\lambda_{2}(1-\lambda_{1}))+ \bar{x}_{1}\bar{x}_{2}(\lambda_{2}^{2}+\lambda_{1}(1-\lambda_{1}))$.  Now we could have instead considered the vector field $\bar{w}_{y}=(0,-y_{2},0,y_{1},0,\ldots,0)$.  Repeating the above arguments with this vector we get at $\bar{x}$ that $\langle [v,w_{y}],\bar{v}_{T_{k}}\rangle = \bar{x}_{1}\bar{x}_{2}(\lambda_{1}\lambda_{2}+\lambda_{2}(1-\lambda_{1}))+ \bar{y}_{1}\bar{y}_{2}(\lambda_{2}^{2}+\lambda_{1}(1-\lambda_{1}))$.  We wish to find $\bar{x}$ where at least one of the two quantities $\langle[v,w],\bar{v}_{T_{k}}\rangle$ or $\langle[v,w_{y}],\bar{v}_{T_{k}}\rangle$ is nonzero.  So assume at some $\bar{x}$ both vanish.  Then if $\langle[v,w],\bar{v}_{T_{k}}\rangle = 0$ and $\langle[v,w_{y}],\bar{v}_{T_{k}}\rangle = 0$ we have that $ (\bar{x}_{1}\bar{x}_{2} +\bar{y}_{1}\bar{y}_{2})(\lambda_{2}^{2}+\lambda_{2} +\lambda_{1}(1-\lambda_{1})) = 0 \Rightarrow \lambda_{2}^{2}+\lambda_{1}(1-\lambda_{1}) = -\lambda_{2} \Rightarrow$ (plugging into first equation) $\lambda_{2}(\bar{x}_{1}\bar{x}_{2}-\bar{y}_{1}\bar{y}_{2})=0$.  Now if we pick $\bar{x}$ such that $\bar{x}_{1}\bar{x}_{2}-\bar{y}_{1}\bar{y}_{2} \neq 0$ then this implies $\lambda_{2} = 0$.  But this implies $\lambda_{1}(1-\lambda_{1}) = 0$, which is a contraction because we showed this quantity is positive.  Hence not both vanish and we have constructed our horizontal vectors.

Conversely if $T^{k}$ is split such that $T^{n}$ is a maximal torus containing $T^{k}$ then we can write $\mathds{R}^{m} = V^{T}_{1}\oplus\ldots\oplus V^{T}_{n}$, where $V^{T}_{i}$ are again an orthogonal decomposition with respect to the $T^{n}$ action, such that we can identify $T^{k}=S^{1}\times S^{1} \ldots \times S^{1}$ where the $i th$ $S^{1}$ factor acts by rotating only $V^{T}_{i}$.  Hence $\mathds{R}^{m}/T^{k}$ is isometric to a product of lines and half lines, and is flat.
\end{proof}

To see how this is useful we need the following

\begin{lemma}\label{lem_tor_split}
Let $T^{k}$ be a split action on $\mathds{R}^{m}$.  Then there exists a local Riemannian reduction for the action at the origin.
\end{lemma}
\begin{proof}
The proof is just a more involved version of Lemma \ref{lem_tor_3}.  Identifying $T^{k}$ with its image in $O(m)$ we can, because $T^{k}$ is split, write $\mathds{R}^{m} = V_{1}\oplus V_{2}\ldots \oplus V_{k}\oplus \mathds{R}^{m-2k}$ and $T^{k}=S^{1}\times S^{1} \ldots \times S^{1}$ where $i$th $S^{1}$ factor acts by rotation on $V_{i}$.  In particular we write $\mathds{R}^{m} = (x_{1},y_{1},\ldots,x_{k},y_{k},z)$ where $V_{i}=span\{x_{i},y_{i}\}$ and $z\in \mathds{R}^{m-2k}$.  Let $S\subseteq \mathds{R}^{m}$ be the subset $S=\{(x_{1},y_{1},\ldots,x_{k},y_{k},z): y_{i}=0\}$.  Let $\Gamma\leq T$ be the subgroup $\Gamma = \mathds{Z}_{2}\times\ldots\times\mathds{Z}_{2}$ where the $i$th $\mathds{Z}_{2}$ acts by rotating $V_{i}$ by an angle of $\pi$.

This action induces a stratification on $\mathds{R}^{m}$ where for each subset $N\subseteq \{1,2,\ldots k\}$ the sets $\mathcal{S}_{N}=\{(x_{1},y_{1},\ldots,x_{k},y_{k},z) : (x_{s},y_{s})=(0,0)$ if $s\in N,(x_{s},y_{s}) \neq (0,0)$ if $s\not\in N \}$ are the strata.  We see as in Lemma \ref{lem_tor_3} that the orbit map $\iota:S/\Gamma\rightarrow \mathds{R}^{m}/T^{k}$ is a homeomorphism and its restriction to each stratum is a diffeomorphism, and so $(S,\Gamma)$ is a local reduction.  We need to show that it is a local Riemannian reduction.  So let $g$ be a smooth metric on $\mathds{R}^{m}$ which is invariant under the torus action $T^{k}$ and let $p_{\mathcal{V}}$ be the projection to the vertical subspace.  We will show the restriction of $p_{\mathcal{V}}$ to $S$ is smooth.

Similar to Lemma \ref{lem_tor_3} let $\mathcal{Y}$ be the smooth distribution on $\mathds{R}^{m}$ which at each point is $\bigoplus_{i}y_{i}$-axis, the span of the $y$-axes. If $x\in S$ then $x$ is in a unique stratum $\mathcal{S}_{N}$.  We see that the isotropy subgroup $I_{x}$ at $x$ is spanned by those $S^{1}$ factors which act on $V_{s}$ for $s\in N$, and that the vertical distribution $\mathcal{V}$ for the $T^{k}$ action at $x$ is $\bigoplus_{s\not\in N}y_{s}$-axis.  By identifying the tangent space at $x$ with $\mathds{R}^{m}$ itself we see that $\forall s\in N$  if $v\in V_{s}$ and $w\not\in V_{s}$, where additionally $w$ has no components in $V_{s}$, there exists an element of the isotropy which fixes $w$ and maps $v$ to $-v$.  Since the action is isometric we must have $g(v,w)=g(-v,w)=0$.  On the other hand since the isotropy contains the full $S^{1}$ rotation group on $V_{s}$ we see the restriction of $g$ to $V_{s}$ is conformal to the standard metric from $\mathds{R}^{m}$.  Hence we see if $v$ points along the $y_{s}$-axis then $g(v,T_{x}S)=0$.  In particular we see that for $w\in T_{x}S$ that $p_{\mathcal{Y}}(w)=p_{\mathcal{V}}(w)$.  Since $x$ was arbitrary we see $p_{\mathcal{Y}}|_{S}=p_{\mathcal{V}}|_{S}$.  But $p_{\mathcal{Y}}$ is smooth, so $p_{\mathcal{V}}$ is smooth.
\end{proof}
\begin{remark}\label{rem_orb_lRr_eq}
Notice that Lemmas \ref{lem_split_curv} and \ref{lem_tor_split} also now show that $p\in \mathds{R}^{m}/T^{k}$ has a local Riemannian orbifold structure iff for any lift $\tilde p\in \mathds{R}^{m}$ there must be a local Riemannian reduction (since in this case the curvatures of $\mathds{R}^{m}/T^{k}$ at $p$ must be bounded, but if there was not a local Riemannian reduction at $\tilde p$ then the isotropy action of $T^{k}$ at $\tilde p$ would not be split and hence the curvatures at $p$ would be unbounded).
\end{remark}

The above also tells us that if the quotient $M/T$ has bounded stratified curvature then it is a Riemannian orbifold, see also \cite{BKLT}.
\section{Proof of Theorems \ref{thm_main} and \ref{thm_main2}}\label{sec_main1}

Using the work of the previous sections it is now possible to complete the proofs of the main theorems.

\begin{proof}[Proof of Theorem \ref{thm_main} and \ref{thm_main2}]
Let $(M^{n}_{i},g_{i},p_{i})\stackrel{GH}{\rightarrow} (X,d,p_{\infty})$ and $p\in X$.  By Theorem \ref{thm_main_collapse} there exists a Riemannian manifold $(\tilde U,g_{\infty})$ with $g_{\infty}$ a $C^{k+1,\alpha}$ metric and a proper and isometric action by $N$, a finite extension of a nilpotent, such that $\tilde U/N$ is isometric to the neighborhood $B_{\epsilon}(p)$ for some small $\epsilon$.  Now if $\tilde p$ is a lift of $p$ to $\tilde U$ then the isotropy $I_{\tilde p}$ has identity component which is a compact nilpotent, and hence is a torus, with $dim I_{\tilde p}<dim N$ by Theorem \ref{thm_main_collapse}.  Then by Theorem \ref{thm_main_torus} we thus have that $dim(B\cap B_{\epsilon}(p))\leq min\{n-5,dim X-3\}$.  Since $p$ was arbitrary we have proved Theorem \ref{thm_main}.

To prove Theorem \ref{thm_main2} assume for $\epsilon$ small that the stratified curvature of $X$ in $B_{\epsilon}(p)$ is uniformly bounded.  Now assume in this case that the induced isotropy action of $I_{\tilde p}$ on $\mathds{R}^{m}$, $m = n - dim N + dim I$, is not split.  Then since the stratified curvature of $X$ is bounded we see that the tangent cone at $p$ is stratified flat (this is a cone geometry so limit is okay), but the tangent cone is also isometric to $\mathds{R}^{m}/I_{\tilde p}$.  By Lemma \ref{lem_split_curv} $\mathds{R}^{m}/I_{\tilde p}$ has unbounded curvature if $I_{\tilde p}$ is not split, and hence $I_{\tilde p}$ must be a split action.  Then in follows from Lemma \ref{lem_tor_split} that there is a local Riemannian reduction for this action, and so a neighborhood of $p$ is isometric to a Riemannian orbifold.  Conversely if $p\in X$ has a neighborhood isometric to a $C^{k+1,\alpha}$ Riemannian orbifold then the result is clear.
\end{proof}
\begin{remark}
To prove Theorem \ref{thm_main} under the assumption that the $(M^{n},g_{i})$ have only Ricci curvature bounds and a lower bound on the conjugacy radius we note that by the results of \cite{A} and \cite{PWY} we can perturb $(M^{n},g_{i})$ to $(M^{n},\tilde g_{i})$ in a $C^{1,\alpha}$ controlled way such that $(M^{n},\tilde g_{i})$ has uniform bounds on its curvature.  Then we can use Theorem \ref{thm_main_collapse} on $(M^{n},\tilde g_{i})$ to locally unwrap with $C^{1,\alpha}$ controlled geometry, which must also be $C^{1,\alpha}$ controlled for the same local unwrapping of $(M^{n},g_{i})$.  Since both unwrappings are invariant under the same fundamental group action we still see that the equivariant Gromov Hausdorff limit of the local unwrappings of $(M^{n},g_{i})$ limits with a finite extension of a nilpotent action as in Theorem \ref{thm_main_collapse}, and we can repeat the arguments of above.  If $(M^{n},g_{i})$ only has a uniform lower bound on the Ricci we can use \cite{AC} and \cite{PWY} to again perturb, though this time the new metric may only be $C^{0,\alpha}$ close to the original.
\end{remark}

\section{Proof of Theorem \ref{thm_main3}}\label{sec_main2}

Finally we are in a position to use the results of \cite{CT} to finish Theorem \ref{thm_main3}

\begin{proof}[Proof of Theorem \ref{thm_main3}]
Because of the bounds on the Ricci curvature we can certainly construct $(X,d,p)$ such that $(M_{i},g_{i},p_{i})\stackrel{GH}{\rightarrow}(X,d,p)$ after passing to a subsequence.  Let $d_{i}$ be distance functions on the disjoint unions $M_{i}\sqcup X$ such that the Hausdorff distances satisfy $\delta_{i}=d_{i,H}(M_{i},X)\rightarrow 0$.  It follows from this $\epsilon$-regularity theorem of \cite{CT} and a standard covering argument that there exists $C>0$ such that for each $i$ there are points $\{q_{ij}\}^{N}_{j=1}\in M_{i}$, where $N\leq N(D)$, such that on $(M_{i},g_{i})$ we have $|Rm_{i}|(y)\leq min\{1,C d(y,\{q_{ij}\})^{-2}\}$ away from the $\{q_{ij}\}$.  For each $q_{ij}$ let $\tilde q_{ij}\in X$ such that $d_{i}(q_{ij},\tilde q_{ij})<2\delta_{i}$.  Now for each $j$ fixed either $\tilde q_{ij}$ tends to infinity or after passing to a subsequence we can limit to a point $\tilde q_{j}$.  If $j$ is such that $q_{ij}$ tends to infinity then we drop it, and then after passing to a subsequence we get $\{\tilde q_{ij}\}\rightarrow \{q_{j}\}^{N}_{1}$ ($N$ maybe smaller than before).

Now let $x\in X$, $d\equiv d(x,\{q_{j}\})$ and $x_{i}\in M_{i}$ be such that $d_{i}(x,x_{i})<2\delta_{i}$.  Then for all $i$ sufficiently large, which certainly depends on $x$, we have that $d(x_{i},\{q_{ij}\})>d/2$.  Hence on $B_{d/4}(x_{i})$ we have that $|Rm_{i}|\leq min\{1,Cd^{-2}\}$ (where $C$ is possibly altered by a factor from before).  By standard estimates that means that on $B_{d/8}(x_{i})$ that $g_{i}$ is $\{A_{d}\}^{\infty}_{0}$ regular, where $A^{i}_{d}=A^{i}d^{-(2+i)}$ and $A^{i}$ only depends on $C$ (and in particular is independent of $x$).  Hence by Theorem \ref{thm_main} and Remark \ref{rem_thm_main} a neighborhood $U$ of $x$ is a smooth Riemannian orbifold.  Lower curvature bounds pass to the limit in the Alexandroff sense, so in particular on the open dense smooth manifold part of $U$ we have the desired lower curvature bounds in the standard sense.  Since the metric on $U$ lifts by a finite cover to a metric on Euclidean space and extends smoothly over the isotropy we see the lower curvature bound holds in the orbifold sense.
\end{proof}

It is interesting to point out that if $(X,d)$ is a collapsed Einstein manifold with singularities at $\{p_{j}\}$ and $(Y,\bar d, O(n))$ is the collapsed frame bundle, so that $Y/O(n)\approx X$, then $Y$ is a smooth manifold away from the orbits above $\{p_{j}\}$ and has a uniform upper $C/d(\{O(n)\cdot p_{j}\},\cdot)^{2}$ bound on the curvatures as well as such a lower bound.  In this context it is then clear that the upper bound on orbifold curvature of $X$ is then controlled by the O'Neill integrability tensor $A$ of the $O(n)$ action on $Y$, in particular by the last comment we have that $|sec_{X}+|A|^{2}|$ blows up at most quadratically near a singular point.  Further when we are near a singular point $p$ of $X$ an upper bound on the sectional curvatures of $X$ can be controlled at least under some additional assumptions on $X$, for instance if the tangent cones at $p$ are noncollapsing.

\appendix
\section{Geometry of Quotients}\label{sec_gq}

We first fix some notation.  Note that our definition of a slice is mildly extended from the usual definition.

\begin{definition}\label{slice}
Let $M^{n}$ be a smooth manifold with $G$ a finite dimensional Lie Group acting smoothly and properly on it and $x\in M$.  Let $G_{x} = \{x\cdot g : g\in G\}$ be the orbit of $x$ and $I_{x}= \{g\in G : x\cdot g = x\}$ be the isotropy subgroup at $x$.  We call a pair $(S_{x},\phi_{x})$ a slice at $x$ if $S_{x}\subseteq M$ is an $I_{x}$-invariant submanifold which is a slice at $x$ in the usual sense and $\phi_{x}:S_{x}\rightarrow \mathds{R}^{m}$ is an $I_{x}$-equivariant diffeomorphism with $\phi_{x}(x)=0$, $m=n-(dim G -dim I_{x})$ and where $I_{x}$ has an $O(m)$ action on $\mathds{R}^{m}$.  We denote the open set $G_{S}\equiv S\cdot G$.
\end{definition}

Throughout this section, unless otherwise stated, $M^{n}$ will denote a smooth manifold with $G$ a finite dimensional Lie Group acting smoothly and properly on it.  As in Definition \ref{slice} it will be convenient for us to a consider a smooth slice as being the pair $(S_{x},\phi_{x})$ as opposed to the more standard notation of just the submanifold $S_{x}$.  As usual if we have a smooth $G$-invariant metric $g$ on $M$ we can construct a slice and a slice map at a point $x\in M$ by considering $S_{x}\equiv \{exp_{x}(v): v\in TG^{\perp}_{x}$ ,$|v|<r \}$ and letting $\phi_{x}$ be a rescaled inverse of the exponential map at $x$ restricted to $G_{x}^{\bot}$ (note if $G$ acts properly then this is a slice for $r$ sufficiently small).  We will later be considering $G$-invariant metrics on $M$ which are not smooth, and hence it will be useful to have a smooth slice $(S_{x},\phi_{x})$ fixed beforehand as opposed to constructing one with respect to the given metric (the exponential map behaves badly on nonsmooth metrics).

Now if $M$ is a smooth manifold with a smooth proper action by $G$ and $(S_{x},\phi_{x})$ is a slice, then we can define a map $\tilde\phi_{x}:\mathds{R}^{m}\times G \rightarrow S_{x}\cdot G = G_{S}$ by $\tilde\phi_{x}(v,g) \equiv \phi_{x}^{-1}(v)\cdot g$.  By construction this map is equivariant with respect to the right $G$-action.  We also have a left $I_{x}$-action on $\mathds{R}^{m}\times G$ by $i\cdot(c,g) = (i^{-1}(v),i\cdot g)$ and with respect to this action we see that $\tilde\phi_{x}(i\cdot(v,g)) = \phi^{-1}_{x}(i^{-1}(v))\cdot ig = \phi^{-1}_{x}(v)\cdot i^{-1} ig = \tilde\phi_{x}(v,g)$.  This tells us that $\tilde\phi_{x}$ is left $I_{x}$-invariant and we have a well defined map $\tilde\phi_{x}: I_{x}\backslash(\mathds{R}^{m}\times G) \rightarrow G_{S}$, and for a proper action this map is a diffeomorphism.  Notice this structure follows uniquely given the slice map $\phi_{x}$, and since we will make use of it it will be convenient to give it a name:

\begin{definition}
We call the map $\tilde\phi_{x}:\mathds{R}^{m}\times G \rightarrow G_{S}$ the total slice map and $\mathds{R}^{m}\times G$ equipped with the right $G$-action and left $I_{x}$-action as before the total space.
\end{definition}

The term total space simply refers to the fact that $\tilde\phi_{x}$ defines an $I_{x}$-principal bundle structure over $G_{S}$.  When $M$ is equipped with a $G$ invariant metric we can put an associated geometry on the total space $\mathds{R}^{m}\times G$:

\begin{lemma}
Let $M$ and $G$ as before with $(S_{x},\phi_{x})$ a smooth slice.  Let $g$ be a $C^{k,\alpha}$ $G$-invariant metric on $M$, then there exists a $C^{k,\alpha}$ metric $\tilde g$ on $\mathds{R}^{m}\times G$ which is invariant under both the right $G$ and left $I_{x}$-actions, and for which the mapping $\tilde\phi_{x}$ is a Riemannian submersion.
\end{lemma}
\begin{proof}
Begin by picking a metric on $G$ which is right invariant and invariant under the left action of $I_{x}$ ($I_{x}$ is compact so this is possible), and let $\bar g$ be the induced smooth product metric on $\mathds{R}^{m}\times G$.  So $\bar g$ is also invariant under the right $G$ and left $I_{x}$-actions, and in particular defines a smooth principal horizontal distribution with respect to the $I_{x}$-action which is also invariant under the right $G$ action.  Then we can alter $\bar g$ to a $C^{k,\alpha}$ metric $\tilde g$ by lifting the metric $g$ on $M$ along the horizontal distribution defined by $\bar g$.  By the construction of this horizontal distribution the metric $\tilde g$ is still invariant under the $G$ and $I_{x}$ actions and now $\tilde\phi_{x}$ is a Riemannian submersion.
\end{proof}

We will sometimes refer to the metric $\tilde g$ from the above as a total metric on $\mathds{R}^{m}\times G$.  A corollary of the above is a geometric version of the slice theorem:

\begin{corollary}\label{cor_geom_slice}
Let $(M^{n},g)$ be a $C^{k,\alpha}$ Riemannian Manifold with $G$ a Lie Group acting faithfully, properly and isometrically on $M$, and let $x\in M$.  Then for $m=n-dim G_{x}+dim I_{x}$ there exists a $C^{k,\alpha}$ metric $h$ on $\mathds{R}^{m}$ and an $O(m)$ action by $I_{x}$ which also acts isometrically with respect to $h$, such that a neighborhood of $x G \in M/G$ is isometric to $(\mathds{R}^{m},h)/I_{x}$.
\end{corollary}
\begin{remark}
Topologically this is just the standard slice theorem, however it is not in general true that if you restrict the metric from $M$ to a slice $S$ that the quotient of $S$ by $I$ will be isometric to a neighborhood of $M/G$, hence the metric $h$ must be constructed by other means.
\end{remark}
\begin{proof}
Let $(S_{x},\phi_{x})$ be a smooth slice at $x$ and let $\tilde g$ be a total metric on the total space $\mathds{R}^{m}\times G$ as in the last lemma.  Because the action is proper and the right action of $G$ and left action by $I$ commute we have that $(\mathds{R}^{m}\times G,\tilde g)/I/G\stackrel{isom}{\approx} (\mathds{R}^{m}\times G,\tilde g)/G/I$.  But $(\mathds{R}^{m}\times G,\tilde g)/I/G \stackrel{isom}{\approx} (G_{S},g)/G$ is exactly isometric to a neighborhood of $x G$ in $M/G$ and $(\mathds{R}^{m}\times G,\tilde g)/G \stackrel{isom}{\approx} (\mathds{R}^{m}, h)$ for some $C^{k,\alpha}$ metric $\tilde h$.  Hence a neighborhood of $x G$ is isometric to $(\mathds{R}^{m}, h)/I$, as claimed.
\end{proof}

Now consider the trivial vector bundle $G_{S}\times\mathfrak{g}$ over $G_{S}$, where $\mathfrak{g}$ is the lie algebra of $G$.  If $M$ is equipped with a $G$-invariant metric $g$ then using a total metric $\tilde g$ on $\mathds{R}^{m}\times G$ we can construct a fiber metric $\tilde h$ on this vector bundle as follows:  Let $x\in G_{S}$ and $\xi, \eta\in\mathfrak{g}$ be elements of the vector bundle above $x$.  Let $\tilde x\in \mathds{R}^{m}\times G$ be a lift of $x$ to the total space and $\tilde e_{*}:\mathfrak{g}\rightarrow T(\mathds{R}^{m}\times G)$ be the pushforward derived from the derivative of the right $G$ action on $\mathds{R}^{m}\times G$.  Then we can define $\tilde h_{x}(\xi,\eta) = \tilde g_{\tilde x}(\tilde e_{*}(\xi),\tilde e_{*}(\eta))$, where we note that this is independent of the lift $\tilde x$ and is a $C^{k,\alpha}$ fiber metric on the vector bundle $G_{S}\times\mathfrak{g}$.  Now there is a vector bundle mapping $e_{*}: G_{S}\times\mathfrak{g}\rightarrow TG_{S}$ where $e_{*}(x,\xi)\in T_{x}G_{S}$ is the pushforward of $\xi$ under the differential of the $G$ action on $G_{S}$ at $x$.  If $\xi \in i_{x}$ (which by definition is the lie algebra of the isotropy group $I_{x}$) then $e_{*}(x,\xi)=0$, and more interestingly at each $x$ if we restrict $e_{*}$ to $i_{x}^{\bot}$, the perpendicular is taken with respect to the metric $\tilde h$, then $e_{*}$ becomes a linear isometry.  Thus the metric $\tilde h$ has given us a nondegenerate extension of the noncontinuous pullback metric from $e_{*}$.

\begin{definition}
We call the fiber metric $\tilde h$ constructed above on $G_{S}\times\mathfrak{g}$ an adjoint metric.
\end{definition}
\begin{remark}
Note that if $G$ acts freely then $\tilde h$ is just the metric induced on the adjoint bundle by a right invariant metric.
\end{remark}

\section{Orbifolds}\label{sec_orb}
The goal of this section is to show the following

\begin{theorem}\label{thm_main_orb}
Let $(O,g)$ and $(U,h)$ be $C^{k,\alpha}$ Riemannian orbifolds, $k\in \mathds{N}$ and $0<\alpha<1$.  Let $\phi:O\rightarrow U$ be a distance preserving homeomorphism.  Then $\phi$ is $C^{k+1,\alpha}$ in the orbifold sense.
\end{theorem}

Recall that being $C^{k+1,\alpha}$ in the orbifold sense implies both the existence of local lifts of $\phi$ to the Euclidean orbifold covers and regularity of this local lift.  Once a local lift is constructed the regularity question is standard, see for instance \cite{CH} and \cite{T}, and so most of this section is dedicated to the construction of such a lift.  Our main use of this theorem is the following proposition.

\begin{prop}\label{finite_ext_quot}
Let $(M,g)$ be a $C^{k,\alpha}$ Riemannian Manifold and let $\tilde{G}$ be a Lie Group acting properly and isometrically on $M$.  Assume $\tilde{G}=G\rtimes F$ where $F$ is a finite group and $(M,g)/G$ is a $C^{k,\alpha}$ Riemannian orbifold.  Then $(M,g)/\tilde{G}$ is a $C^{k,\alpha}$ Riemannian orbifold.
\end{prop}
\begin{proof}
Since $G$ is a normal subgroup of $\tilde{G}$ we see that there is an induced action of $F$ by distance preserving homeomorphisms on $M/G$ and that $(M,g)/\tilde{G}$ is isometric to $(M/G)/F$.  Let $x\in M/G$ and $F_{x}$ the isotropy group of $F$ at $x$.  Because $F$ is discrete we can pick $r>0$ such that $\forall f\in F$ if $B_{r}(f\cdot x)\cap B_{r}(x)\neq \emptyset$ then $f\in F_{x}$, thus $B_{r}(xF)$ in $(M/G)/F$ is isometric to $B_{r}(x)/F_{x}$.  After possibly making $r$ smaller we can assume that $B_{r}(x)$ is isometric to $(\mathds{R}^{m},\tilde{g})/H$ where $\tilde{g}$ is a $C^{k,\alpha}$ metric on $\mathds{R}^{m}$ and $H\leq O(m)$ is a finite group which acts isometrically with respect to $\tilde{g}$. We have by Theorem \ref{thm_main_orb} that each $f\in F_{x}$ lifts to a $\tilde f$ which is an isometric action on $(\mathds{R}^{m},\tilde{g})$.  Let $y\in \mathds{R}^{m}_{reg}$, where $\mathds{R}^{m}_{reg}$ is the open dense subset that $H$ acts freely on.  Then each lift $\tilde{f}$ of $f$ is well defined by the value $\tilde{f}(y)$ (this follows because as in Proposition \ref{prop_orb_2} we have that $\tilde f$ maps $\mathds{R}^{m}_{reg}$ to itself and so the differential of $\tilde{f}$ at $y$ is determined by $f$, since $\tilde{f}$ is an isometry this determines $\tilde{f}$).  We see that for each $h\in H$ that $\tilde{f}(h\cdot x)$ and $h\cdot\tilde{f}(x)$ are also lifts of $f$.  Hence for each $h_{1}\in H$ there must exist a unique $h_{2}\in H$ such that $\tilde{f}(h_{1}\cdot x) = h_{2}\cdot\tilde{f}(x)$ (namely  pick $h_{2}$ such that $\tilde{f}(h_{1}\cdot y) = h_{2}\cdot\tilde{f}(y)$, then this must hold for all $x$ since both are lifts).  So if we consider the collection $\tilde F$ of all lifts of all elements of $F$ then this is exactly the statement that $H$ is a normal subgroup of $\tilde F$ and $\tilde{F} = H\rtimes F$.  Hence  $\mathds{R}^{m}/\tilde{F} =(\mathds{R}^{\tilde{m}}/H)/F$ and the quotient is a Riemannian orbifold.
\end{proof}

We begin by recalling some basic points about the geometry of Riemannian Orbifolds.  The Riemannian metric on the orbifold induces a stratified Riemannian metric on the orbifold as a stratified space, hence we can define a natural length space distance function on the orbifold.  At each point it is the case that there is a neighborhood such that this distance function agrees with the quotient distance function from a chart.  We call a curve a smooth geodesic if in a neighborhood of each point there is a lift of the curve to a smooth geodesic.  Such a smooth geodesic curve is uniquely defined by its value and tangent vector at a point (where of course the tangent space is the quotient of $\mathds{R}^{n}$ by the local group at that point).  A smooth geodesic need not be locally distance minimizing (in fact, it is local minimizing near a point iff in a neighborhood of that point it lies in a unique stratum), but it is at least always locally minimizing in one direction.  That is is say, if $\gamma(t)$ is a smooth geodesic then $\exists$ $\epsilon>0$ such that $\gamma|_{[0,\epsilon]}$ and $\gamma|_{[-\epsilon,0]}$ are segments, even though $\gamma|_{[-\epsilon,\epsilon]}$ may not be.  It can also be shown that a segment is a smooth geodesic.

Theorem \ref{thm_main_orb} will follow immediately once we establish the following local version

\begin{prop}\label{prop_orb_1}
Let $(\mathds{R}^{n},g_{1})$ $(\mathds{R}^{n},g_{2})$ be $C^{k,\alpha}$ and let $\Gamma_{1}$, $\Gamma_{2}\leq O(n)$ be finite subgroups which act isometrically on $g_{1},g_{2}$, respectively.  Let $\phi:(\mathds{R}^{n},g_{1})/\Gamma_{1} \rightarrow (\mathds{R}^{n},g_{2})/\Gamma_{2}$ be a distance preserving homeomorphism with $\phi(0)=0$.  Then $\phi$ is $C^{k+1,\alpha}$ in the orbifold sense and $\Gamma_{1}$, $\Gamma_{2}$ are conjugate.
\end{prop}

The first step is to establish the above for standard quotient geometries:

\begin{prop}\label{prop_orb_2}
Let $\Gamma_{1}$,$\Gamma_{2}\leq O(n)$ be finite subgroups.  Let $\phi:\mathds{R}^{n}/\Gamma_{1} \rightarrow \mathds{R}^{n}/\Gamma_{2}$ be a distance preserving homeomorphism.  Then there $\exists$ an affine isometric lift $\tilde{\phi}:\mathds{R}^{n}\rightarrow \mathds{R}^{n}$ of $\phi$ and $A\in O(n)$ such that $A\Gamma_{1}A^{-1} =
\Gamma_{2}$.
\end{prop}

Given this we will quickly prove Proposition \ref{prop_orb_1} and hence the Theorem \ref{thm_main_orb}.

\begin{proof}[Proof of Proposition \ref{prop_orb_1}]
Let $\phi:(\mathds{R}^{n},g_{1})/\Gamma_{1} \rightarrow (\mathds{R}^{n},g_{2})/\Gamma_{2}$ be a distance
preserving homeomorphism.  We show that $\phi$ has a $C^{k+1,\alpha}$ lift at $0$, the proof is the same at any other point.

Let $r>0$ be such that $\widetilde{exp}_{i}:\tilde{B}_{r}(0)\subseteq T_{0}\mathds{R}^{n}\rightarrow (B_{r}(0),g_{i})\subseteq (\mathds{R}^{n},g_{i})$ are homeomorphisms.  Note that $\widetilde{exp}_{i}$ are equivariant with respect to the $\Gamma_{i}$ actions and so descend to homeomorhphism $exp_{i}:\tilde{B}_{r}/\Gamma_{i} \rightarrow (B_{r},d_{i})$.  Now $\phi$ maps segments to segments, hence $\forall$ $v\in \mathds{R}^{n}/\Gamma_{1}$ if we let $\gamma_{v}$ be the unique geodesic in $(\mathds{R}^{n},g_{1})/\Gamma_{1}$ with $\gamma_{v}(0)=0$ and $\dot{\gamma_{v}}=v$ and we let $\epsilon$ be sufficiently small, then $\phi(\gamma_{v}|_{[0,\epsilon]})$ is a geodesic segment in $(\mathds{R}^{n},g_{2})/\Gamma_{2}$ beginning at $0$ with some tangent vector $w$ at $0$.  We define the map $\Lambda:\mathds{R}^{n}/\Gamma_{1} \rightarrow \mathds{R}^{n}/\Gamma_{2}$ by $\Lambda(v)=w$. If we restrict $\phi$ to $B_{r}$ we see that $\phi(x)=exp_{2}(\Lambda(exp^{-1}_{1}(x)))$ for each $x\in B_{r}$.  We will first show $\Lambda$ is a distance preserving homeomorphism.  Let $\epsilon_{j}\rightarrow 0$ and $\phi_{j}=\phi$.  Then we see that $\phi_{j}$ is a distance preserving homeomorphism from $(B_{r/\epsilon_{j}},d_{1}/\epsilon_{j})$ to $(B_{r/\epsilon_{j}},d_{2} /\epsilon_{j})$ and that $\phi_{j}=exp^{j}_{2}(\Lambda((exp^{j}_{1})^{-1}(x)))$, where $exp^{j}$ are exponential maps with respect to the rescaled  metrics (and so differ from the original exponential maps by just by a rescaling). Letting $j$ tend to infinity we see by Ascolli that after passing to a subsequence $\phi_{j}$ converges to a distance preserving homeomorphism $\phi_{\infty}:\mathds{R}^{n}/\Gamma_{1} \rightarrow \mathds{R}^{n}/\Gamma_{2}$ and that $exp_{i}$ tend to the identity maps.  Hence $\Lambda$ is a distance preserving homeomorphism.  By proposition \ref{prop_orb_2} it lifts to a smooth map $\tilde{\Lambda}:\mathds{R}^{n}\rightarrow \mathds{R}^{n}$.  Since the exponential maps lift, we see that a composition of the lifts is a lift of the compositions and hence $\phi$ lifts to a map $\tilde\phi:\mathds{R}^{n}\rightarrow\mathds{R}^{n}$.

Since $\tilde\phi$ is a lift of a distance preserving map, then at least on the open dense subset $\mathds{R}^{n}_{reg}\subseteq\mathds{R}^{n}$ on which the $\Gamma_{i}$ act freely on we see $\tilde\phi$ is locally distance preserving.  But the distance functions on $(\mathds{R}^{n},g_{i})$ are length spaces induced from a continuous metric, and so being locally distance preserving on $\mathds{R}^{n}_{reg}$ implies $\tilde\phi$ is a distance preserving map on all $\mathds{R}^{n}$, and hence $C^{k+1,\alpha}$.
\end{proof}

To construct the lift in Proposition \ref{prop_orb_2} we first construct a lift of the regular part of $(\mathds{R}^{n}/\Gamma)$. This lift must be done in two steps, first by showing a lift to each component of $\mathds{R}^{n}_{reg}$ is possible and then showing that the lifts to each of the components may be done in a compatible way, so that the lift to $\mathds{R}^{n}_{reg}$ extends continuously to a lift on $\mathds{R}^{n}$.

Recall that if a group $H$ acts on a manifold $M$ we call an open subset $U\subset M$ a fundamental domain for $H$ if $\forall h_{1}\neq h_{2}\in H$ $h_{1}U\cap h_{2}U=\emptyset$ and $\cup h \overline{U}=M$.

\begin{lemma}
Let $\Gamma\leq O(n)$ be finite subgroup.  Let $\mathds{R}^{n}_{reg}=\cup\mathds{R}^{n}_{reg,i}$
be the open dense subset that $\Gamma$ acts freely on with $\mathds{R}^{n}_{reg,i}$ the connected components.  Let $R\leq\Gamma$ be the reflection subgroup.  Then each $\mathds{R}^{n}_{reg,i}$ is a fundamental domain for $R$.
\end{lemma}
\begin{proof}
Let $\{r_{i}\}\in R$ be the reflections.  Then $\mathds{R}^{n}-\cup_{r_{i}}Fix(r_{i})=\cup C_{i}$, where $C_{i}$ are connected open simplex cones and $Fix(a)$, for $a\in O(n)$, is the linear subspace of fix points of $a$. It is well known (see [GB]) that these $C_{i}$ are fundamental domains for $R$.  Now $\mathds{R}^{n}_{reg} = \mathds{R}^{n} - \cup_{a\in\Gamma}Fix(a) = \mathds{R}^{n} - \cup_{a\not\in \{r_{i}\}}Fix(a) - \cup_{r_{i}}Fix(r_{i})$.  However if $a\in\Gamma$ is not a reflection, then the fix point set of $a$ has codimension at least $2$. Hence $\mathds{R}^{n}_{reg,i} = C_{i} - (\cup_{a\not\in \{r_{i}\}}Fix(a))\cap C_{i}$, so is open dense in $C_{i}$ and is a fundamental domain for $R$ as well.
\end{proof}

As a consequence of the above we see that the connected components of $\mathds{R}^{n}_{reg}$ are in a $1-1$ correspondence with the elements of $R$, and in fact $R$ acts freely and transitively on the them.  That is, for each $r\in R$ and each $\mathds{R}^{n}_{reg,i}$ we see that $r(\mathds{R}^{n}_{reg,i})=\mathds{R}^{n}_{reg,j}$ for $i\neq j$, and for any distinct $\mathds{R}^{n}_{reg,i},\mathds{R}^{n}_{reg,j}$ $\exists !$ $r\in R$ such that
$r(\mathds{R}^{n}_{reg,i})=\mathds{R}^{n}_{reg,j}$.  Also note that the reflection
subgroup $R\leq\Gamma$ is always a normal subgroup.

\begin{lemma}
Let $\Gamma\leq O(n)$ be a finite subgroup with $\mathds{R}^{n}_{reg}=\cup\mathds{R}^{n}_{reg,i}$
and $R\leq\Gamma$ the normal reflection subgroup.  Then $\pi_{1}(\mathds{R}^{n}_{reg,i})= \mathds{Z}^{k}$ and $\pi_{1}((\mathds{R}^{n}/\Gamma)_{reg})=\mathds{Z}^{k}\rtimes(\Gamma/R)$
for some $k\in\mathds{N}$.
\end{lemma}
\begin{proof}
We saw before that $\mathds{R}^{n}_{reg,i} = C_{i} - (\cup_{a\not\in \{r_{i}\}}Fix(a))\cap C_{i}$. Since $Fix(a)$ are subspaces of codimension at least $2$ and $C_{i}$ is just a simplex cone, we see that $\pi_{1}(\mathds{R}^{n}_{reg,i})= \mathds{Z}^{k}$ where $k$ is the number of such subspaces of codimension exactly $2$.  Since $R$ acts freely and transitively on $\mathds{R}^{n}_{reg,i}$ we see that $\mathds{R}^{n}_{reg,i}\approx \mathds{R}^{n}_{reg}/R$.  Hence since $\Gamma/R$ acts freely on
$\mathds{R}^{n}_{reg}/R$ we see that $\pi_{1}((\mathds{R}^{n}/\Gamma)_{reg}) = \pi_{1}((\mathds{R}^{n}_{reg}/\Gamma))
= \pi_{1}((\mathds{R}^{n}_{reg}/R)/(\Gamma/R)) = \mathds{Z}^{k}\rtimes(\Gamma/R)$
\end{proof}

The above will be enough to construct a lift to each component $\mathds{R}^{n}_{reg,i}$. To lift to each component in a globally compatible way we will need the following:

\begin{lemma}
Let $R_{1},R_{2}\leq O(n)$ be finite reflection groups.  Let $D_{1},D_{2}\subseteq \mathds{R}^{n}$ be convex fundamental domains for $R_{1},R_{2}$, respectively, and $\phi:D_{1}\rightarrow D_{2}$ be a distance preserving
homeomorphism.  Then $\exists!$ isometric affine extension $\tilde{\phi}:\mathds{R}^{n}\rightarrow\mathds{R}^{n}$
and isomorphism $\tilde{\phi}:R_{1}\rightarrow R_{2}$ such that $\forall r\in R$ and $x\in \mathds{R}^{n}$ we have $\tilde{\phi}(rx)=\tilde{\phi}(r)\tilde{\phi}(x)$.  Further, $R_{1}$ and $R_{2}$ must be conjugate.
\end{lemma}

The existence of the extension in the above is easy to show, the useful part is the existence of the homomorphism.  This tells us that when we restrict to the open dense subsets $\tilde{\phi}: R_{1}\cdot D_{1}\rightarrow R_{2}\cdot D_{2}$ that $\tilde{\phi}$ is a lift of $\phi$.

\begin{proof}
$D_{i}$ convex $\Rightarrow$ $\phi(x)=Ax+b$ for $A\in O(n)$ and $b\in \mathds{R}^{n}$, so the unique extension is obvious.  For $i=1,2$ let $\mathds{R}^{n} - \cup_{a_{i}\in R_{i}} Fix(a_{i})= \cup C_{i}^{j}$ where $C_{i}^{j}$ are the connected cones which are the natural fundamental domains for $R_{i}$.  Because $D_{i}$ are convex, $D_{i}$ are connected.  Since $D_{i}\cap Fix(a_{i}) = \emptyset$ we must then have that there is a $C_{1}\in \{C_{1}^{j}\}$ and a $C_{2}\in \{C_{2}^{j}\}$ such that $D_{i}\subseteq C_{i}$.  Since $D_{i}$ are fundamental domains, they must be open dense in $C_{i}$.

Let $\{r_{1}^{j}\}$ be reflections in $R_{1}$, $\{r_{2}^{j}\}$ in $R_{2}$, such that
$\partial C_{i} = \partial\overline{D}_{i}=\cup_{\{r_{i}^{j}\}} Fix(r_{i}^{j})$.  Note that $\{r_{1}^{j}\}$ generates $R_{1}$, $\{r_{2}^{j}\}$ generates $R_{2}$ (see \cite{GB}).  Now $\phi$ extends to a distance preserving homeomorphism $\phi:\overline{C}_{1}\rightarrow \overline{C}_{2}$.  So $b=\phi(0)\in \overline{C}_{2}$.  If $\phi(x)=0$ then $x\in\overline{C}_{1}$ since $0\in\overline{C}_{2}$, and  hence $x=-A^{-1}b\in\overline{C}_{1}$.  Since $\overline{C}_{1}$ is a cone we see that $2x=-2A^{-1}b\in\overline{C}_{1}$ $\Rightarrow$ $\phi(2x)=-b\in\overline{C}_{2}$.  Hence the whole line generated by $b$ must be contained in $\overline{C}_{2}$.  Since $\overline{C}_{1}$ is a cone we have that $\forall x\in \overline{C}_{1}$ that $x-A^{-1}b\in \overline{C}_{1}$, and so $Ax=\phi(x-A^{-1}b)$ is a distance preserving homeomorphism $A:\overline{C}_{1}\rightarrow\overline{C}_{2}$.

Now $\forall r_{1}^{j}$, $A r_{1}^{j} A^{-1}$ is another reflection uniquely defined by $Fix(A r_{1}^{j} A^{-1}) = A\cdot Fix(r_{1}^{j})$.  Since $A$ is a linear homeomorphism $A:\overline{C}_{1}\rightarrow\overline{C}_{2}$, $A$ must map hypersurface boundary components of $\overline{C}_{1}$ to hypersurface boundary components of $\overline{C}_{2}$.  Hence $A r_{1}^{j} A^{-1}\in \{r_{2}^{j}\}$.  Since $\{r_{i}^{j}\}$ generate $R_{i}$ we see that $\tilde{\phi}(r)\equiv ArA^{-1}$ maps $R_{1}$ into $R_{2}$. By using the same argument on $\tilde\phi^{-1}=A^{-1} r A$ we see this map is an isomorphism.  Hence $R_{1}$ and $R_{2}$ are conjugate.

Next we show that $ArA^{-1}b =b$ $\forall r\in R_{1}$.  By the previous paragraph it is enough to show $r_{2}^{j}b=b$ $\forall j$.  So assume not for some $j$.  Then the line generated by $b$ is not contained in $Fix(r_{2}^{j})$ and so must lie on both sides of the hypersurface.  However, $\overline{C}_{2}$ lies strictly on one side, since $Fix(r_{2}^{j})$ forms boundary component, and the line generated by $b$ is contained in $\overline{C_{2}}$ by the above.  This is a contradiction.  So $r_{2}^{j}b=b$ $\forall j$ and hence $ArA^{-1}b =b$ $\forall r\in R_{1}$.  Thus we see that $\forall r\in R_{1}$ and $x\in\mathds{R}^{n}$ that $\tilde{\phi}(r\cdot x)=ArA^{-1}x+b=ArA^{-1}(Ax+b)= \tilde{\phi}(r)\tilde{\phi}(x)$.
\end{proof}

We are now in a position to finish the proof of proposition \ref{prop_orb_2}:

\begin{proof}[Proof of Proposition \ref{prop_orb_2}]
Let $\phi:\mathds{R}^{n}/\Gamma_{1}\rightarrow \mathds{R}^{n}/\Gamma_{2}$ be a distance preserving homeomorphism with $R_{1}\leq\Gamma_{1}$, $R_{2}\leq\Gamma_{2}$ be the reflection subgroups.  Now $\phi$ maps segments to segments, and the interior of each segment is contained in a single stratum.  Since the strata are convex we then see $\phi$ maps stratum to stratum and since $\phi$ is a homeomorphism it must restrict to a distance preserving homeomorphism $\phi:(\mathds{R}^{n}/\Gamma_{1})_{reg}\rightarrow (\mathds{R}^{n}/\Gamma_{2})_{reg}$.

Pick $D_{1}\in \{\mathds{R}^{n}_{reg,1,i}\}$ and $D_{2}\in \{\mathds{R}^{n}_{reg,2,i}\}$.  Then we showed that

\[
\pi_{1}(D_{1})=\mathds{Z}^{k_{1}}, \pi_{1}((\mathds{R}^{n}/\Gamma_{1})_{reg})=\mathds{Z}^{k_{1}}\rtimes(\Gamma_{1}/R_{1})
\]
\[
\pi_{1}(D_{2})=\mathds{Z}^{k_{2}}, \pi_{1}((\mathds{R}^{n}/\Gamma_{2})_{reg}) =\mathds{Z}^{k_{2}}\rtimes(\Gamma_{2}/R_{2})
\]

Since $\phi$ is a homeomorphism from $(\mathds{R}^{n}/\Gamma_{1})_{reg}$ to $(\mathds{R}^{n}/\Gamma_{2})_{reg}$ we see that $k_{1}=k_{2}\equiv k$ and $|\Gamma_{1}/R_{1}|=|\Gamma_{2}/R_{2}|$.  If $p_{i}:D_{i}\rightarrow (\mathds{R}^{n}/\Gamma_{i})_{reg}$ are the projection maps, then we can lift on the left to get $\tilde{\phi}=\phi\circ p_{1}:D_{1}\rightarrow (\mathds{R}^{n}/\Gamma_{2})_{reg}$.  Now $p_{1},p_{2},\phi$ all induce corresponding maps between the fundamental groups of the manifolds.  Since $p_{i}$ are normal coverings $p_{i}(\pi_{1}(D_{i}))=\mathds{Z}^{k}$.  Since $\pi_{1}((\mathds{R}^{n}/\Gamma_{1})_{reg})$, $\pi_{1}((\mathds{R}^{n}/\Gamma_{2})_{reg})$ are finite extensions of $\mathds{Z}^{k}$, the induced homomorphism between the fundamental groups by $\phi$ must map the $\mathds{Z}^{k}$ factor to itself.  Hence $\tilde{\phi}(\pi_{1}(D_{1}))\subseteq p_{2}(\pi_{1}(D_{2}))$ (in fact equals).  So there is a lifted distance preserving homeomorphism $\tilde{\phi}:D_{1}\rightarrow D_{2}$.

We apply the last lemma to see that $\exists$ $\tilde{\phi}=Ax+b:\mathds{R}^{n}\rightarrow\mathds{R}^{n}$ which is a lift
on $R_{1}\cdot D_{1}\rightarrow R_{2}\cdot D_{2}$, hence $\tilde{\phi}$ is a lift of $\phi$.  Since $R_{1}$ and $R_{2}$ are conjugate and $|\Gamma_{1}/R_{1}|=|\Gamma_{2}/R_{2}|$ we see $|\Gamma_{1}|=|\Gamma_{2}|$.  Let $U\subseteq \mathds{R}^{n}_{reg,1}\cap\tilde{\phi}^{-1}(\mathds{R}^{n}_{reg,2})$ be an open connected subset (possible since $\mathds{R}^{n}_{reg,1}$,$\mathds{R}^{n}_{reg,2}$ are open dense cones and $\tilde{\phi}$ affine).  Since $\tilde{\phi}$ is a lift we have $\forall t_{1}\in\Gamma_{1}$ and $x\in\mathds{R}^{n}$ $\exists$ $t_{2}(x)\in \Gamma_{2}$ such that $\tilde{\phi}(t_{1}x)=t_{2}(x)\tilde{\phi}(x)$.  Then for $t_{1}$ fixed we see that $t_{2}(x)\tilde{\phi}(x)$ is continuous.  On $U$ this is possible iff $t_{2}(x)=t_{2}$ is constant.  Hence on $U$ $\tilde{\phi}(t_{1}x)=t_{2}\tilde{\phi}(x)$.  Since $\tilde{\phi}$ affine this must hold on all $\mathds{R}^{n}$. Hence $A(t_{1}(A^{-1}(x-b)))+b=t_{2}x$ $\Rightarrow$ $At_{1}A^{-1}x+(b-AbA^{-1})=t_{2}x$.  Letting $x$ go to zero
we see that $(b-AbA^{-1})=0$ and hence $A\Gamma_{1}A^{-1}\subseteq \Gamma_{2}$.  Since $|\Gamma_{1}|=|\Gamma_{2}|$ we see $A\Gamma_{1}A^{-1} = \Gamma_{2}$.
\end{proof}

\section*{Acknowledgements} The authors would like the thank John Lott for his helpful corrections.

\end{document}